\documentclass[11pt]{article}

\usepackage{bbm}
\usepackage{amsmath, amssymb, amsthm, xypic, stmaryrd,graphicx}
\usepackage[all]{xy}
\usepackage{tikz-cd}
\usepackage{boxedminipage}
\usepackage{mathtools}
\DeclareMathOperator{\supp}{supp}

\DeclareMathOperator{\diam}{diam}

\DeclareMathOperator{\vol}{vol}
\DeclareMathOperator{\alg}{alg}
\DeclareMathOperator{\red}{red}
\DeclareMathOperator{\loc}{loc}

\DeclareMathOperator{\Av}{Av}

\DeclareMathOperator{\ind}{index}

\DeclareMathOperator{\End}{End}

\DeclareMathOperator{\U}{U}
\DeclareMathOperator{\Spin}{Spin}

\DeclareMathOperator{\Pen}{Pen}

\DeclareMathOperator{\pt}{pt}

\DeclareMathOperator{\restr}{restr}

\newcommand{\beq}[1]{\begin{equation} \label{#1}}
\newcommand{\eeq}{\end{equation}}

\newcommand{\bea}{\begin{eqnarray}}
\newcommand{\eea}{\end{eqnarray}}

\begin{document}

\theoremstyle{plain}
\newtheorem{theorem}{Theorem}[section]
\newtheorem{thm}{Theorem}[section]
\newtheorem{lemma}[theorem]{Lemma}
\newtheorem{proposition}[theorem]{Proposition}
\newtheorem{prop}[theorem]{Proposition}
\newtheorem{corollary}[theorem]{Corollary}
\newtheorem{conjecture}[theorem]{Conjecture}

\theoremstyle{definition}
\newtheorem{definition}[theorem]{Definition}
\newtheorem{defn}[theorem]{Definition}
\newtheorem{example}[theorem]{Example}
\newtheorem{remark}[theorem]{Remark}
\newtheorem{rem}[theorem]{Remark}

\newcommand{\C}{\mathbb{C}}
\newcommand{\R}{\mathbb{R}}
\newcommand{\Z}{\mathbb{Z}}
\newcommand{\N}{\mathbb{N}}

\newcommand{\Supp}{{\rm Supp}}

\newcommand{\field}[1]{\mathbb{#1}}
\newcommand{\bZ}{\field{Z}}
\newcommand{\bR}{\field{R}}
\newcommand{\bC}{\field{C}}
\newcommand{\bN}{\field{N}}
\newcommand{\bT}{\field{T}}
\newcommand{\cB}{{\mathcal{B} }}
\newcommand{\cK}{{\mathcal{K} }}
\newcommand{\cF}{{\mathcal{F} }}
\newcommand{\cO}{{\mathcal{O} }}
\newcommand{\cE}{\mathcal{E}}
\newcommand{\cS}{\mathcal{S}}
\newcommand{\calL}{\mathcal{L}}

\newcommand{\HH}{{\mathcal{H} }}
\newcommand{\tilH}{\widetilde{\HH}}
\newcommand{\HX}{\HH_X}
\newcommand{\Hpi}{\HH_{\pi}}
\newcommand{\HHpi}{\HH \otimes \HH_{\pi}}
\newcommand{\Ltwopi}{L^2_{\pi}(X, \HHpi)}

\newcommand{\KK}{K\!K}

\newcommand{\D}{D \hspace{-0.27cm }\slash}
\newcommand{\Dsmall}{D \hspace{-0.19cm }\slash}

\newcommand{\mybigwedge}{\textstyle{\bigwedge}}

\newcommand{\CGmax}{C^*_{G, \max}}
\newcommand{\DGmax}{D^*_{G, \max}}
\newcommand{\CGred}{C^*_{G, \red}}
\newcommand{\CGalg}{C^*_{G, \alg}}
\newcommand{\DGalg}{D^*_{G, \alg}}
\newcommand{\CGker}{C^*_{G, \ker}}
\newcommand{\Cmax}{C^*_{\max}}
\newcommand{\Dmax}{D^*_{\max}}
\newcommand{\Cred}{C^*_{\red}}
\newcommand{\Calg}{C^*_{\alg}}
\newcommand{\Dalg}{D^*_{\alg}}
\newcommand{\Cker}{C^*_{\ker}}
\newcommand{\tilCalg}{\widetilde{C}^*_{\alg}}
\newcommand{\Cpiker}{C^*_{\pi, \ker}}
\newcommand{\Cpialg}{C^*_{\pi, \alg}}
\newcommand{\Cpimax}{C^*_{\pi, \max}}

\newcommand{\one}{\mathbbm{1}}

\newcommand{\Avpi}{\Av^{\pi}}

\newcommand{\Gtc}{\Gamma^{\infty}_{tc}}
\newcommand{\tilD}{\widetilde{D}}
\newcommand{\XoneH}{X^{\HH}_1}

\newcommand{\XX}{\mathfrak{X}}

\def\kt{\mathfrak{t}}
\def\kk{\mathfrak{k}}
\def\kp{\mathfrak{p}}
\def\kg{\mathfrak{g}}
\def\kh{\mathfrak{h}}

\newcommand{\pilamrho}{[\pi_{\lambda+\rho}]}

\newcommand{\Trestr}{\mathcal{T}_{\restr}}
\newcommand{\TdN}{\mathcal{T}_{d_N}}

\newcommand{\omG}{\om/\hspace{-1mm}/G}
\newcommand{\om}{\omega} \newcommand{\Om}{\Omega}

\newcommand{\QcwR}{quantization commutes with reduction}

\newcommand{\Spinc}{\Spin^c}

\def\kt{\mathfrak{t}}
\def\kk{\mathfrak{k}}
\def\kp{\mathfrak{p}}
\def\kg{\mathfrak{g}}
\def\kh{\mathfrak{h}}

\newcommand{\ddt}{\left. \frac{d}{dt}\right|_{t=0}}

\newenvironment{proofof}[1]
{\noindent \emph{Proof of #1.}}{\hfill $\square$}


\title{Equivariant Callias index theory via coarse geometry}
\author{Hao Guo,\footnote{Texas A\&M University, \texttt{haoguo@math.tamu.edu}}
{}
Peter Hochs\footnote{University of Adelaide, \texttt{peter.hochs@adelaide.edu.au}}
 {} 
 and Varghese Mathai\footnote{University of Adelaide, \texttt{mathai.varghese@adelaide.edu.au}}}

\maketitle

\begin{center}
{\it Dedicated to the memory of John Roe}
\end{center}

\begin{abstract}
The equivariant coarse index is well-understood and widely used for actions by discrete groups. We extend the definition of this index to general locally compact groups. We use a suitable notion of admissible modules over $C^*$-algebras of continuous functions  to obtain a meaningful index. Inspired by work by Roe, we then develop a localised variant, with values in the $K$-theory of a group $C^*$-algebra. This generalises the Baum--Connes assembly map to non-cocompact actions. We show that an equivariant index for Callias-type operators is a special case of this localised index, obtain results on existence and non-existence of Riemannian metrics of positive scalar curvature invariant under proper group actions, and show that a localised version of the Baum--Connes conjecture is weaker than the original conjecture, while still giving a conceptual description of the $K$-theory of a group $C^*$-algebra.  
\end{abstract}

\tableofcontents

\section{Introduction}

\subsection*{Background}

\emph{Coarse geometry} is the study of large-scale structures of metric spaces. Important invariants in this area are various versions of the \emph{Roe algebra} and their $K$-theory groups. The \emph{coarse index}, with values in such $K$-theory groups, is a powerful tool that has been studied and applied by many authors. A standard introduction is \cite{Roe96}. A central problem is the coarse Baum--Connes conjecture \cite{Roe93}, which has a range of important consequences.
Important areas of applications of coarse index theory are obstructions to Riemannian metrics of positive scalar curvature (see \cite{Schick14} for a survey) and the Novikov conjecture (see for example \cite{Yu98, Yu00}). 

Equivariant versions of the Roe algebra and the coarse index have been developed for proper actions by discrete groups. Another refinement is a localised variant of the coarse index developed by Roe \cite{Roe16}, generalising an index defined by Gromov and Lawson \cite{Gromov83}. This localised coarse index applies, in a certain precise sense, to operators that are invertible outside a given subset of the space. For actions by discrete groups, an equivariant version of this localised index theory in terms of coarse $K$-homology was recently developed by Bunke and Engel \cite{Bunke18}.

 Equivariant coarse index theory for general locally compact groups would be useful in the study of such groups and their actions. In particular, a localised approach to this index, which takes values in $K$-theory of the group $C^*$-algebra, offers greater flexibility compared to the standard equivariant index for actions with compact quotients \cite{Connes94}. However, the topology of a non-discrete group poses some technical challenges in the development of equivariant coarse index theory for such groups.

\subsection*{Results}

Our main goal in this paper is to develop equivariant coarse index theory for proper actions by general locally compact groups $G$, and in particular a localised version with values in $K_*(C^*_{\red}(G))$. Here $C^*_{\red}(G)$ is the reduced group $C^*$-algebra of $G$. Our secondary goal is to demonstrate the usefulness of this theory by showing how it simultaneously generalises other versions of index theory, and by obtaining applications to Riemannian metrics of positive scalar curvature invariant under proper actions. Our main results are:
\begin{enumerate}
\item constructions of equivariant Roe algebras and an equivariant coarse index, and in particular localised versions of these objects (Definitions \ref{def Roe alg}, \ref{def loc Roe}, \ref{def coarse index ell} and \ref{def loc index ell});
\item a proof that the analytic assembly map \cite{Connes94} and the equivariant index of Callias-type operators in \cite{Guo18} are special cases of the localised equivariant coarse index (Theorem \ref{thm Callias} and Corollary \ref{cor loc index ass map});
\item obstructions to and existence of Riemannian metrics with positive scalar curvature, invariant under proper group actions (Proposition \ref{prop psc} and Theorem \ref{thm:induction});
\item the formulation of a localised version of the surjectivity part of the Baum--Connes conjecture \cite{Connes94}, and relations between the localised and original conjectures (Conjecture \ref{conj LBCC} and Proposition \ref{prop LBCC}).
\end{enumerate}
In forthcoming work \cite{GHM2},  we will give further applications of the index theory we develop in this paper, showing that it refines  the indices in \cite{Braverman14, Mathai13}, and obtaining an application to the quantisation commutes with reduction problem.

\subsection*{The localised equivariant coarse index}

Let $M$ be a complete Riemannian manifold, on which a locally compact group $G$ acts properly and isometrically.
Let $D$ be an elliptic differential operator on $M$. For the definition of the localised index, we assume that there is a $G$-invariant subset $Z \subset M$ such that $D^2$ has a uniform positive lower bound outside $Z$. Then we obtain the localised equivariant coarse index
\[
\ind_G^Z(D) \in K_*(C^*(Z)^G).
\]
(See Definition \ref{def loc index ell}.) Here $C^*(Z)^G$ is the equivariant Roe algebra of $Z$. We are particularly interested in the case where $Z/G$ is compact, so that $C^*(Z)^G$ is stably isomorphic to $C^*_{\red}(G)$. While there is no technical reason a priori to restrict to the case where $Z/G$ is compact, this special case is interesting for several reasons, namely, in this case: 
\begin{enumerate}
\item the localised equivariant coarse index and the $K$-theory group it lies in are independent of $Z$;
\item the receptacle of the index, namely $K_*(C^*_{\red}(G))$, is a rich and relevant object (in particular, nonzero); there exist many tools to extract information from it, such as traces and higher cyclic cohomology classes on (smooth subalgebras of) $C^*_{\red}(G)$; 
\item various existing indices, including the analytic assembly map, are special cases, as we discuss below.
\end{enumerate}

Operators $D$ to which the localised equivariant index applies include the following three important special cases:
\begin{enumerate}
\item Callias-type operators of the form $D = \tilde D + \Phi$, where $\tilde D$ is a Dirac operator and $\Phi$ is a vector bundle endomorphism making $D^2$ uniformly positive outside $Z$. The study of these operators, their indices and their applications was initiated by Callias \cite{Callias78}, and extended in various directions by many authors, see e.g.\ \cite{Anghel89,Anghel93, Bott78, Braverman18, Bruening92, Bunke95, Carvalho14, Cecchini16, Kottke11, Kottke15, Kucerovsky01, Wimmer14}. The equivariant case for proper actions was treated in \cite{Guo18};
\item Dirac operators $D$ whose curvature term $R$ in the Weitzenb\"ock formula $D^2 = \Delta + R$ is uniformly positive outside $Z$ \cite{Gromov83, Roe16};
\item Dirac operators $D$ on manifolds with boundary that are invertible on the boundary, extended to cylinders attached to these boundaries.
\end{enumerate}
For Callias-type operators, the coarse-geometric approach in this paper may already be useful in the case of trivial groups. In addition, we can now consider the lift of a (non-equivariant) Callias-type operator on a manifold $M$ to the universal cover of $M$, and obtain an equivariant index in the $K$-theory of the fundamental group of $M$. This is a more refined invariant than the Fredholm index of the initial operator. 

In the case where $Z/G$ is compact, the localised equivariant index generalises the Baum--Connes assembly map from the case of actions with compact quotients to the cases above. This allows us to formulate a localised version of the surjectivity direction of the Baum--Connes conjecture. We will show that this localised surjectivity is implied by standard Baum--Connes surjectivity.

One of the technical challenges in constructing a meaningful index in this context is to develop the appropriate notion of an \emph{admissible module}. For actions by discrete groups, this was done in Definition 2.2 in \cite{Yu10}. The Roe algebra of a metric space $X$ acted on by a locally compact group $G$ is defined in terms of operators a on Hilbert space $H_X$ with compatible actions by $C_0(X)$ and $G$. The resulting algebra should ideally be independent of the choice of $H_X$, and its $K$-theory should contain relevant information about $G$, and possibly also $X$. (A natural initial choice would be $H_X = L^2(X)$ for a Borel measure on $X$, but this does not contain enough information  if, for example, $X$ is a point or if $G$ is compact and acts trivially on $X$.) We achieve these two things by taking $H_X$ to be an admissible module, in the sense that we define in Sections \ref{sec prelim} and \ref{sec proofs Roe}. We indicate how Roe algebras and the associated $K$-theory groups and indices
 defined in terms of admissible modules on the one hand, and more geometric, but non-admissible modules on the other, are related in Subsection \ref{sec non adm}.
 
\subsection*{Outline of this paper}

We introduce admissible modules and the associated Roe algebras in Section \ref{sec prelim}. We use these notions in Section \ref{sec index} to define the equivariant coarse index and its localised version. In Section \ref{sec results} we apply these notions to show that the equivariant Callias-type index in \cite{Guo18} is a special case, and we establish results on positive scalar curvature, as well as state a localised Baum--Connes conjecture. Proofs of the properties of admissible modules and Roe algebras from Section \ref{sec prelim} are given in Section \ref{sec proofs Roe}. Proofs of the results in Section \ref{sec results} are given in Sections \ref{sec pf Callias} and \ref{sec appl}.







\subsection*{Acknowledgements}

The authors are grateful to Rufus Willett, Zhizhang Xie, and Guoliang Yu for their helpful advice. Varghese Mathai was supported by funding from the Australian Research Council, through the Australian Laureate Fellowship FL170100020. Hao Guo was supported in part by funding from the National Science Foundation under Grant No. 1564398.

\section{Equivariant Roe algebras and admissible modules} \label{sec prelim}

A key idea in this paper is to use coarse geometry and Roe algebras to construct a localised equivariant index for proper actions with values in the $K$-theory of a group $C^*$-algebra. We start by discussing the necessary background in coarse geometry. Much of the material in this section is well-known in the case of discrete groups, but we will see that the generalisation to general locally compact groups requires some work.

Throughout this section, $(X,d)$ will denote a proper metric space, i.e.\ a metric space in which closed balls are compact, and $G$ a locally compact group acting properly and isometrically on $X$. We assume $G$ to be unimodular and fix a Haar measure $dg$ on $G$. Throughout this paper, we will use left and right invariance of $dg$ without mentioning this explicitly, in arguments involving substitutions in integrals over $G$.
(The contents of this paper can likely be generalised to non-unimodular group if the modular function is inserted where appropriate.)
We will sometimes assume $X/G$ to be compact, but not always. We always view $L^2(G)$ as a unitary representation of $G$ via the left-regular representation. 

The two properties of the modules and algebras we define here that are most important to the construction of the equivariant localised coarse index in Section \ref{sec index} are Theorems \ref{thm exist adm} and \ref{thm Roe group Cstar}. These are proved in Subsections \ref{sec pf thm exist} and \ref{sec kernels}, respectively.

\subsection{Admissible modules} \label{sec adm}

We will construct the reduced and maximal equivariant Roe algebras of $X$ in terms of \emph{admissible $C_0(X)$-modules}, and a particularly useful type of such modules we call \emph{geometric admissible modules}. Using admissible modules ensures that the algebras constructed are independent of the choice of module. Their purpose is also to ensure that the Roe algebras we use contain sufficient information about the group $G$, as illustrated in Examples \ref{ex non adm} and \ref{ex ind non adm}. Admissible modules were first defined in the case of discrete groups by Guoliang Yu in Definition 2.2 in \cite{Yu10}. For non-discrete $G$, the definition needs to take into account the topology of $G$.

Admissible modules are special cases of ample, equivariant $C_0(X)$-modules.
\begin{definition} \label{def C0X module}
An \emph{equivariant $C_0(X)$-module} is a Hilbert space $H_X$ with a unitary representation of $G$, together with a $*$-homomorphism
\[
\pi\colon C_0(X) \to \cB(H_X)
\]
such that for all $g \in G$ and $f \in C_0(X)$,
\[
\pi(g\cdot f) = g \circ \pi(f) \circ  g^{-1}.
\]
Here $g\cdot f$ is the function mapping $x \in X$ to $f(g^{-1}x)$.

An equivariant $C_0(X)$-module is \emph{nondegenerate} if $\pi(C_0(X))H_X$ is dense in $H_X$. It is \emph{standard} if $\pi(f)$ is a compact operator only if $f =0$. The module is \emph{ample} if it is nondegenerate and standard.
\end{definition}
We will usually omit the homomorphism $\pi$ from the notation, and write $f \cdot \xi := \pi(f)\xi$ for $f \in C_0(X)$ and $\xi \in H_X$.

\begin{example}
The space $H_G=L^2(G)\otimes H$, for a separable infinite-dimensional Hilbert space $H$ equipped with the trivial $G$-representation, is an ample equivariant $C_0(G)$-module with respect to the multiplicative action of $C_0(G)$ on $L^2(G)$.
\end{example}

The action by $C_0(X)$ on any ample, equivariant $C_0(X)$-module $H_X$ has a unique extension to an action by the algebra $L^{\infty}(X)$ of bounded Borel functions, characterised by the property that for a uniformly bounded sequence in $L^{\infty}(X)$ converging pointwise, the corresponding operators on $H_X$ converge strongly. All functions we will apply this extension to are bounded, continuous functions on closed sets in $X$, such as the indicator function $\one_{Y}$ of a closed subset $Y \subset X$.

\begin{definition}
	Let $H_X$ be an ample $C_0(X)$-module. An operator $T \in \cB(H_X)$ has \emph{finite propagation} if there is an $r>0$ such that for all $f_1, f_2 \in C_0(X)$ whose supports are further than $r$ apart, we have $f_1 T f_2 = 0$. 
	
	An operator $T \in \cB(H_X)$ is \emph{locally compact} if for all $f \in C_0(X)$, the operators $fT$ and $Tf$ are compact.
\end{definition}

We now formulate the general notion of admissible module over a space.
\begin{definition}
\label{def adm mod}
Let $G$ is a locally compact group. Let $H_X$ be an ample, equivariant $C_0(X)$-module. If $X/G$ is compact, $H_X$ is said to be \emph{admissible} if there is a $G$-equivariant unitary isomorphism
\[
\Psi:L^2(G) \otimes H\cong H_X ,
\]
for a separable infinite-dimensional Hilbert space $H$ equipped with the trivial $G$-representation, such that for any bounded, $G$-equivariant operator $T$ on $H_X$,
\begin{enumerate}
	\item $T$ has finite propagation with respect to the action by $C_0(X)$ if and only if $\Psi^{-1}\circ T \circ \Psi$ has finite propagation with respect to the action by $C_0(G)$;
	\item $T$ is locally compact with respect to the action by $C_0(X)$ if and only if $\Psi^{-1}\circ T \circ \Psi$ is locally compact (with respect to the action by $C_0(G)$.
\end{enumerate}

If $X/G$ is not necessarily compact, then an ample, equivariant $C_0(X)$-module is admissible if for every closed, $G$-invariant subset $Y \subset X$ such that $Y/G$ is compact, $\one_Y H_X$ is an admissible module over $C_0(Y)$.
\end{definition}
\begin{remark}
If $G=\Gamma$ is a discrete group, Definition \ref{def adm mod} is an alternative to the definition of admissible modules given in  the conditions in Definition \ref{def adm mod} are implied by the definition of admissible modules given in \cite{Yu10}. In particular, if $H_X$ is an admissible module in the sense of Definition 2.2 in \cite{Yu10}, then it is also admissible in our sense.  
\end{remark}
\begin{remark}
The idea behind admissible modules is to provide a class of general spaces on which to carry out analysis of operators on $X$, but which also allow us to define an equivariant index of these operators. For example, we will use Conditions 1 and 2 of Definition \ref{def adm mod} in an essential way\footnote{See the proof of Proposition \ref{prop dense kernels}.} to define the equivariant coarse index and the localised equivariant coarse index in $K_*(C^*(G))$ (see Subsections \ref{sec ind maps cocpt} and \ref{sec ind maps noncocpt}).
\end{remark}
In this paper, we will mostly work with a particular geometric type of admissible module.

Consider a covering of $X$ by sets of the form $G \times_{K_j} Y_j$, for compact subgroups $K_j < G$ and compact, $K_j$-invariant slices $Y_j \subset X$. Suppose that the intersections between these sets have measure zero. For each $j$, fix a $K$-invariant measure $dy_j$ on $Y_j$. Together with the Haar measure $dg$, they induce a $G$-invariant measure $dx$ on $X$. We call such a measure \emph{induced from slices}. Such measures are natural choices; see for example Lemma 4.1 in \cite{HSIII}.
\begin{theorem} \label{thm exist adm}
Suppose $X/G$ is compact. Suppose that at least one of the sets $G/K$ and $X/G$ is infinite. Let $H_X = L^2(E) \otimes L^2(G)$, for a Hermitian $G$-vector bundle $E \to X$, defined with respect to the a measure on $X$ induced from slices. Then $H_X$, equipped with the diagonal representation of $G$ and the $C_0(X)$-action on the factor $L^2(E)$ by pointwise multiplication, is an admissible $C_0(X)$ module. 
\end{theorem}
\begin{definition}
If $X/G$ is compact, then $L^2(E) \otimes L^2(G)$, for a Hermitian $G$-vector bundle $E \to X$, defined with respect to the a measure on $X$ induced from slices is a \emph{geometric admissible $C_0(X)$-module}.
\end{definition}

Note that the difference between a general admissible module in Definition \ref{def adm mod} and a geometric module is that on a geometric admissible module, the group $G$ acts diagonally, whereas on a general admissible module in the form $L^2(G)\otimes H$, $G$ acts only on the first factor. The advantage of working with a geometric module is that the action by $C_0(X)$ is explicit. We will in fact usually work with geometric admissible modules. 

The notion of a geometric admissible module, and hence Theorem \ref{thm exist adm}, is a key component of our construction of (localised) coarse indices of elliptic operators in Subsections \ref{sec ind ell} and \ref{sec loc ind ell}.


\begin{remark} \label{rem H inf dim}
The condition in Definition \ref{def adm mod} that $H$ is infinite-dimensional, and the corresponding condition in Theorem \ref{thm exist adm} that $G/K$ or $X/G$ is infinite, are assumed to ensure that the equivariant coarse index theory of Section \ref{sec index} is rich enough to capture information about the group $G$ (see Examples \ref{ex non adm} and \ref{ex ind non adm}). More specifically, infinite-dimensionality of $H$ guarantees that the algebras $D^*(X)^G$, which are used to define the coarse index (see Subsection \ref{sec loc ind}) exist and have the properties needed to define a useful index. It also implies that the localised Roe algebra is independent of the choice of admissible module, as in \eqref{eq loc Roe group Cstar}. (Although for finite-dimensional $H$, the factor $\mathcal{K}$ in \eqref{eq loc Roe group Cstar} would become a finite-dimensional matrix algebra, which makes no difference at the level of $K$-theory.)

In Theorem \ref{thm exist adm}, if both $G/K$ and $X/G$ are finite, then one can still form the admissible module $L^2(E) \otimes L^2(G) \otimes l^2(\N)$.
\end{remark}

\subsection{Equivariant Roe algebras} \label{sec nonloc Roe}

Fix an equivariant $C_0(X)$-module $H_X$.
We denote the algebra of $G$-equivariant bounded operators  $H_X$ by $\cB(H_X)^G$.
\begin{definition}\label{def Roe alg}
The \emph{algebraic equivariant Roe algebra for $H_X$} of $X$ is the algebra  $C^*_{\alg}(X; H_X)^G$  consisting of the locally compact operators in $\cB(H_X)^G$ with finite propagation.
The \emph{equivariant Roe algebra for $H_X$}  of $X$ is the closure $C^*(X; H_X)^G$ of $C^*_{\alg}(X)^G$ in $\cB(H_X)$.

If $H_X$ is an admissible module and $X/G$ is compact, then $C^*_{\alg}(X)^G := C^*_{\alg}(X; H_X)^G$ is the  \emph{algebraic equivariant Roe algebra of $X$}, and $C^*(X)^G := C^*(X; H_X)^G$ is the  \emph{equivariant Roe algebra of $X$}.
\end{definition}
%

\begin{theorem}\label{thm Roe group Cstar}
If $X/G$ is compact and $H_X$ is , then $C^*_{\alg}(X)^G$ is $*$-isomorphic to a dense subalgebra of $C^*(G) \otimes \cK$, where $C^*(G)$ denotes either the reduced or maximal group $C^*$-algebra, and $\cK$ is the algebra of compact operators on a separable Hilbert space.
\end{theorem}
This theorem will be proved in Subsection \ref{sec kernels}. This involves realising Roe algebras in terms of Schwartz kernels of operators.

There is also a maximal version of the equivariant Roe algebra.
For any $*$-algebra $A$, and any $a \in A$, we write
\[
\|a\|_{\max}:= \sup_{\pi} \|\pi(a)\|_{\cB(H_{\pi})},
\]
where the supremum runs over all irreducible $*$-representations $\pi$ of $A$ in Hilbert spaces $H_{\pi}$. This supremum may be infinite. Since this maximal norm is always finite for group $C^*$-algebras, Theorem \ref{thm Roe group Cstar} has the following consequence.
\begin{corollary}
If $X/G$ is compact, then $\|a\|_{\max}<\infty$ for all $a \in C^*_{\alg}(X)^G$.
\end{corollary}
\begin{definition}
If $X/G$ is compact, then the \emph{maximal equivariant Roe algebra} of $X$ is the completion of $C^*_{\alg}(X)^G$ in the maximal norm $\|\cdot \|_{\max}$.
\end{definition}
\begin{remark}
Gong, Wang and Yu \cite{GWY} proved that if $G = \Gamma$ is discrete, then the maximal norm on $C^*_{\alg}(X)^{\Gamma}$ is always finite, even if $X/\Gamma$ is not compact, as long as $X$ has bounded geometry. We expect this to be true also when $G$ is locally compact. (In \cite{GWY}, the Roe algebras are defined in terms of kernels; see Subsection \ref{sec kernels}.)
\end{remark}

If $X/G$ is compact then Theorem \ref{thm Roe group Cstar} implies that there are $*$-isomorphisms
\begin{align}
C^*(X)^G&\cong C^*_{\red}(G) \otimes \cK; \label{eq red Roe cocpt}\\
C^*_{\max}(X)^G&\cong C^*_{\max}(G) \otimes \cK. \label{eq max Roe cocpt}
\end{align}
\begin{remark}
\label{rem Roe indep}
The relations \eqref{eq red Roe cocpt} and \eqref{eq max Roe cocpt} in particular imply that the reduced and maximal algebras are independent of the choice of the admissible $C_0(X)$-module $H_X$. We expect this to be true even if $X/G$ is not compact. In the case of $G=\Gamma$ discrete and $X/\Gamma$ compact, a proof can be found in the forthcoming book \cite{WillettYu}.
\end{remark}

\begin{example} \label{ex non adm}
Suppose that $G = K$ is compact, and $X = \pt$ is a point. Then $l^2(\N)$ is an ample module over $C_0(\pt) = \C$. It is equivariant if we equip it with the trivial action by $K$. The algebraic, reduced and maximal  equivariant Roe algebras defined with respect to this module all equal $\cK(l^2(\N))$, which contains no group-theoretic information about $K$. This is because the module $l^2(\N)$ is not admissible.
\end{example}

\subsection{Localised Roe algebras} \label{sec loc Roe}

The localised index that we will define in Definition \ref{def loc index} involves a localised version of Roe algebras.

Let $H_X$ be an equivariant $C_0(X)$-module.
Let $Z \subset X$ be a $G$-invariant, closed subset.
\begin{definition} \label{def loc Roe}
An operator $T \in \cB(H_X)$ is \emph{supported near $Z$} if there is an $r>0$ such that for all $f \in C_0(X)$ whose support is at least a distance $r$ away from $Z$, we have $fT = Tf = 0$.

The \emph{algebraic equivariant Roe algebra for $H_X$ of $X$, localised at $Z$}, denoted by $C^*_{\alg}(X; Z, H_X)^G$, consists of the operators in $C^*_{\alg}(X, H_X)^G$ supported near $Z$.

The \emph{equivariant Roe algebra for $H_X$ of $X$, localised at $Z$}, denoted by $C^*(X; Z, H_X)^G$, is the closure of $C^*_{\alg}(X; Z, H_X)^G$ in $\cB(H_X)$.


If $Z/G$ is compact, then we call $C^*_{\alg}(X; H_X)^G_{\loc} := C^*_{\alg}(X; Z, H_X)^G$ the \emph{localised algebraic equivariant Roe algebra for $H_X$ of $X$}, and $C^*(X; H_X)^G_{\loc} := C^*(X; Z, H_X)^G$ the \emph{localised equivariant Roe algebra for $H_X$ of $X$}. If $H_X$ is an admissible module, then we omit it from the notation and terminology, and obtain the \emph{localised algebraic equivariant Roe algebra} $C^*_{\alg}(X)^G_{\loc}$ and the  \emph{localised equivariant Roe algebra} $C^*(X)^G_{\loc}$ of $X$.
\end{definition}
Of the terms in Definition \ref{def loc Roe}, the {localised equivariant Roe algebra} $C^*(X)^G_{\loc}$ is the one we are most interested in. Note that if $Z/G$ is {compact}, then the algebras $C^*_{\alg}(X; Z, H_X)^G$ and  $C^*(X; Z, H_X)^G$ are independent of $Z$, as long as $Z/G$ is compact.

For $r>0$, we write
\beq{eq def Pen}
\Pen(Z, r) := \{x\in X; d(x, Z) \leq r\}.
\eeq
In terms of these sets, we have
\beq{eq loc Roe inj lim}
\begin{split}
C^*_{\alg}(X; Z, H_X)^G &= \varinjlim_{r} C^*_{\alg}(\Pen(Z, r); H_X)^G;\\
C^*(X; Z, H_X)^G &= \varinjlim_{r} C^*(\Pen(Z, r); H_X)^G.
\end{split}
\eeq
If $H_X$ is admissible, then the algebra $C^*_{\alg}(\Pen(Z, r))^G$ has a well-defined maximal norm by Theorem \ref{thm Roe group Cstar}, for all $r$.  Hence, by \eqref{eq loc Roe inj lim}, so does $C^*_{\alg}(X)^G_{\loc}$. 
\begin{definition}
If $H_X$ is admissible, the \emph{maximal localised  equivariant Roe algebra of $X$}, denoted by $C^*_{\max}(X)^G_{\loc}$, is the completion of $C^*_{\alg}(X)^G_{\loc}$ in the maximal norm $\|\cdot \|_{\max}$.
\end{definition}

Theorem \ref{thm Roe group Cstar} and \eqref{eq loc Roe inj lim} imply that
\beq{eq loc Roe group Cstar}
\begin{split}
C^*(X)^G_{\loc} & \cong C^*_{\red}(G) \otimes \cK;\\
C^*_{\max}(X)^G_{\loc} & \cong C^*_{\max}(G) \otimes \cK.
\end{split}
\eeq

\section{The localised equivariant index} \label{sec index}


\subsection{Indices of abstract operators}  \label{sec loc ind}

The (non-localised) equivariant coarse index is defined completely analogously to the case for discrete groups, but with Roe algebras defined in terms of the admissible modules from Subsection \ref{sec adm}.

Let $H_X$ be an equivariant $C_0(X)$-module. Let $C^*(X; H_X)^G$ denote the reduced or maximal version (if it exists\footnote{For a general space, one first needs to show finiteness of the maximal norm to know that $C^*_\textnormal{max}(X;H_X)^G$ is well-defined.}) of the equivariant Roe algebra for $H_X$. Let $D^*(X)^G$
be any $C^*$-algebra containing $C^*(X; H_X)^G$ as a two-sided ideal. For example, we can take $D^*(X)^G$ to be the multiplier algebra of $C^*(X; H_X)^G$, or the $C^*$-algebra generated by  $C^*(X; H_X)^G$ and a single operator in $\cB(H_X)^G$. 

\begin{remark}
\label{rem Dstarred}
In the reduced case, a natural choice for $D^*(X)^G$ is the algebra $D^*_{\red}(X)^G$, an equivariant version of the algebra $D^*(X)$ used in \cite{Roe16}. This is defined as the closure in $\cB(H_X)$ of the algebra of operators $T \in \cB(H_X)^G$ with finite propagation such that $[T,f]$ is compact for all $f \in C_0(X)$.
\end{remark}
Let
\beq{eq bdry}
\partial\colon K_{*+1}(D^*(X)^G/C^*(X)^G) \to K_{*}(C^*(X)^G)
\eeq
be the boundary map associated to the short exact sequence
\[
0 \to C^*(X)^G\to D^*(X)^G \to D^*(X)^G/C^*(X)^G \to 0.
\]
\begin{definition} \label{def coarse index}
Let $F \in D^*(X)^G$, and suppose that $F - F^*$ and $F^2-1$ lie in $C^*(X; H_X)^G$.
\begin{itemize}
\item If no grading on $H_X$ given, consider the projection $P = \frac{1}{2}(F+1)$ in $D^*(X)^G/C^*(X; H_X)^G$ and the class $[P] \in K_{0}(D^*(X)^G/C^*(X; H_X)^G)$. Then the \emph{equivariant coarse index} of $F$ is
\[
\ind_G(F) = \partial[P] \quad \in K_1(C^*(X; H_X)^G).
\]
\item If a grading on $H_X$ is given that is preserved by $C_0(X)$ and $G$, and interchanged by $F$, let $F_+$ be the restriction of $F$ to the even-degree part of $H_X$. Then $F_+$ is invertible modulo $C^*(X)^G$, and we have $[F_+] \in K_{1}(D^*(X)^G/C^*(X; H_X)^G)$. The \emph{equivariant coarse index} of $F$ is
\[
\ind_G(F) = \partial[F_+] \quad \in K_0(C^*(X; H_X)^G).
\]
\end{itemize}
\end{definition}
More generally, we will also write $\ind_G$ for the boundary map \eqref{eq bdry}.

Let $Z \subset X$ be a closed $G$-invariant subset. Let $D^*(X)^G \subset \cB(H_X)^G$ be a $C^*$-algebra containing $C^*(X; Z, H_X)^G$ as a two-sided ideal. The algebra $D^*_{\red}(X)^G$ defined above has this property.
\begin{definition} \label{def loc index}
Let $F \in D^*(X)^G$, and suppose that $F - F^*$ and $F^2-1$ lie in $C^*(X; Z, H_X)^G$. The \emph{equivariant coarse index of $F$, localised at $Z$},
\[
\ind_G^Z(F) \in K_*(C^*(X; Z, H_X)^G)
\]
 is defined analogously to $\ind_G(F)$ in Definition \ref{def coarse index}, with $C^*(X, H_X)^G$ replaced by $C^*(X; Z, H_X)^G$ everywhere.
 
 If $Z/G$ is compact and $H_X$ is an admissible module, we write
 \[
 \ind_G^{\loc}(F) := \ind_G^Z(F) \quad \in K_*(C^*(X)^G_{\loc}) = K_*(C^*(G)),
 \]
 and call this the \emph{localised equivariant coarse index} of $F$. Here $C^*(G)$ denotes either the reduced or maximal group $C^*$-algebra.
\end{definition}
We will also denote the boundary map \eqref{eq bdry}, with $C^*(X, H_X)^G$ replaced by $C^*(X; Z, H_X)^G$, by $\ind_G^Z$, or by $\ind_G^{\loc}$ if $Z/G$ is compact and $H_X$ is an admissible module.

\subsection{The equivariant coarse index of elliptic operators} \label{sec ind ell}

In the rest of this section, we work with the reduced version of the equivariant Roe algebra and specialise to the geometric setting that we are most interested in.  Let $X = M$ be a Riemannian manifold, and $d$ the Riemannian distance. Suppose, as before, that $G$ acts properly and isometrically on $M$. Let $E\to M$ 
be a $G$-equivariant Hermitian vector bundle. Let $D$ be a $G$-equivariant,  first order elliptic differential operator on $E$ that is essentially self-adjoint on $L^2(E)$. 

To apply Definitions \ref{def coarse index} and \ref{def loc index} to $D$, we embed $L^2(E)$ into the (geometric) admissible module $H_M := L^2(E) \otimes L^2(G)$ of Theorem \ref{thm exist adm}.   We will illustrate why using admissible modules is necessary in Example \ref{ex ind non adm} (see also Example \ref{ex non adm}).
 
Let $\chi \in C^{\infty}(M)$ be a cutoff function, in the sense that its support has compact intersections with all $G$-orbits, and that for all $m \in M$,
\beq{eq def cutoff}
\int_G \chi(gm)^2\, dg = 1.
\eeq
The map
\beq{eq def j}
j\colon L^2(E)\to H_M,
\eeq
given by
\[
(j(s))(m,g) = \chi(g^{-1}m)s(m),
\]
for $s \in L^2(E)$, $m \in M$ and $g \in G$, is a $G$-equivariant, isometric embedding. It intertwines the actions by $C_0(M)$ on $L^2(E)$ and $H_M$. 
Define the maps
\beq{eq def plus 01}
\oplus\,0, \oplus\,1\colon \cB(L^2(E)) \to \cB(H_M)
\eeq
by identifying operators on $L^2(E)$ with operators on $j(L^2(E))$ via conjugation by $j$, and extending them by zero or the identity operator, respectively, on the orthogonal complement of $j(L^2(E))$ in $H_M$. 



Let $b \in C_b(\R)$ be a normalising function, i.e.\ an odd function with values in $[-1,1]$ such that $\lim_{x \to \infty}b(x)=1$.
Let $D^*(M; L^2(E))^G$ be a unital $*$-subalgebra of $\cB(L^2(E))$ containing $b(D)$, and $C^*(M; L^2(E))^G$ as a two-sided 
ideal. Similarly, let  $D^*(M)^G$ be a unital $*$-subalgebra of $\cB(H_M)$ containing $C^*(M; H_M)^G$ as a two-sided ideal. 
Suppose that the image of $D^*(M; L^2(E))^G$ under the map $\oplus\,1$ lies inside $D^*(M)^G$. 
This is the case if we use multiplier algebras, or the algebra $D^*_{\red}(M)^G$ as in Remark \ref{rem Dstarred}.

If $F \in D^*(M; L^2(E))^G$ and $F^*-F$ and $F^2-1$ lie in $C^*(M; L^2(E))^G$, then $F \oplus\,1 \in D^*(X)^G$ by assumption, and $(F\oplus 1)^*-F\oplus 1$ and $(F\oplus 1)^2-1$ lie in $C^*(M; H_M)^G$. Hence $F \oplus 1$ has an index in $K_*(C^*(M; H_M)^G)$ as in Definition \ref{def coarse index}.
\begin{definition} \label{def coarse index ell}
The \emph{equivariant coarse index} of $D$ is
\[
\ind_G(D) := \ind_G(b(D) \oplus 1) \quad \in K_*(C^*(M; H_M)^G),
\]
where the index on the right hand side is as in Definition \ref{def coarse index}.
\end{definition}
As in Definition \ref{def coarse index}, $\ind_G(D)$ lies in even or odd $K$-theory depending on the presence of a grading on $E$ with respect to which $D$ is odd.

\subsection{The localised index of elliptic operators}\label{sec loc ind ell}

Again, let $Z \subset M$ be a closed,  $G$-invariant subset. 
Suppose that there is a constant $c>0$ such that for all $s \in \Gamma^{\infty}_c(E)$ supported outside $Z$,
\[
\|Ds\|_{L^2} \geq c \|s\|_{L^2}.
\]
Let $b \in C^{\infty}(\R)$ be an odd, increasing function taking values in $\{\pm 1\}$ on $\R \setminus [-c,c]$. 
Form the operator $b(D)$ by functional calculus.
The following result by Roe is the basis of the index theory we develop in this paper.
\begin{proposition} \label{prop Roe}
The operator $b(D) \in \cB(L^2(E))$ satisfies
\[
b(D)^2-1 \in C^*(M; Z, L^2(E))^G.
\]
\end{proposition}
See Lemma 2.3 in \cite{Roe16}. As above Definition \ref{def coarse index ell} This proposition implies that $b(D)\oplus 1$ satisfies the conditions of Definition \ref{def loc index}.
%
\begin{definition} \label{def loc index ell}
The \emph{equivariant coarse  index of $D$, localised at $Z$} is
\beq{eq loc index Z}
\ind_G^Z(D) := \ind_G^Z(b(D) \oplus 1) \quad \in K_*(C^*(X; Z, H_M)^G),
\eeq
where the index on the right hand side is as in Definition \ref{def loc index}.

If $Z/G$ is compact, then the  index \eqref{eq loc index Z} is by definition the \emph{localised equivariant coarse  index of $D$}, and denoted by
\[
\ind_G^{\loc}(D) := \ind_G^{\loc}(b(D) \oplus 1) \quad \in K_*(C^*_{\red}(G)).
\]
\end{definition}
As before, $\ind_G^Z(D)$ and $\ind_G^{\loc}(D)$ lie in even or odd $K$-theory depending on the presence of a grading.

The localised equivariant coarse index of an elliptic operator is the object we are most interested in here. It is a natural generalisation of the Baum--Connes analytic assembly map \cite{Connes94} from cocompact to non-cocompact actions; see Corollary \ref{cor loc index ass map}.

\subsection{Admissible and non-admissible modules} \label{sec non adm}

Let us clarify the relevance of using admissible modules in the definition of the equivariant coarse index. This will lead to
 an equivalent definition of the localised equivariant coarse index in the graded case,  \eqref{eq alt def loc ind} below.

The map $\oplus\,0$ preserves finite propagation, local compactness and having support near $Z$, and hence restricts to an injective $*$-homomorphism
\beq{eq incl plus 0}
\oplus\,0\colon C^*(M; Z, L^2(E))^G \to C^*(M; Z, H_M)^G.
\eeq
We denote the map induced on $K$-theory by $\oplus\,0$ as well.

Viewing $C^*(M; Z, L^2(E))^G$ as a subalgebra of $C^*(M; Z, H_M)^G$ via the map \eqref{eq incl plus 0}, 
we find that the map $\oplus\,1$ descends to a multiplicative but nonlinear map
\beq{eq plus 1 L2}
\oplus\,1\colon D^*(M; L^2(E))^G/C^*(M; L^2(E))^G \to D^*(M)^G/C^*(M; H_M)^G.
\eeq
This induces a homomorphism on odd $K$-theory, which we still denote by $\oplus\,1$.

%
%
%
%
\begin{lemma}\label{lem comm index}
The following diagram commutes:
\[
\xymatrix{
K_1\bigl(D^*(M; L^2(E))^G/C^*(M; Z, L^2(E))^G\bigr) \ar[r]^-{\partial} 
\ar[d]_-{\oplus\,1}& K_{0}(C^*(M; Z, L^2(E))^G)\ar[d]^-{\oplus\,0} \\
K_1(D^*(M)^G/C^*(M; Z, H_M)^G) \ar[r]^-{\partial} & K_{0}(C^*(M; Z, H_M)^G),
}
\]
where the maps $\partial$ are boundary maps in the respective six-term exact sequences.
\end{lemma}
\begin{proof}
	The boundary map can be described explicitly as follows. Suppose that $u$ is an invertible element in $D^*(M; L^2(E))^G/C^*(X; Z, L^2(E))^G$, and let $v$ be its inverse. Let $U$ and $V$ respectively be representatives of $u$ and $v$ in $D^*(M; L^2(E))^G$. Define
	\[
	W=\begin{pmatrix}1&U\\0&1\end{pmatrix}\begin{pmatrix}1&0\\-V&1\end{pmatrix}\begin{pmatrix}1&U\\0&1\end{pmatrix}\begin{pmatrix}0&-1\\1&0\end{pmatrix}.
	\]
	Then 
\[	
\partial[u] = 
\left[W\begin{pmatrix}1&0\\0&0\end{pmatrix}W^{-1}\right]-\left[\begin{pmatrix}1&0\\0&0\end{pmatrix}\right]\in K_0(C^*(M; Z, L^2(E))^G).
\]
Using the analogous description of $\ind_G([u]\oplus 1)$, one can explicitly prove the claim. 
\end{proof}

Let $D$ be as in Subsection \ref{sec loc ind ell}, and suppose that $E$ has a $\Z_2$ grading with respect to which $D$ is odd. Let
\[
\ind_G^{Z, L^2(E)}(D) := \partial [b(D)] \quad \in  K_0(C^*(M; Z, L^2(E))^G)
\]
be the image of $[b(D)]  \in K_1\bigl(D^*(M; L^2(E))^G/C^*(M; Z, L^2(E))^G\bigr)$ under the boundary map.
Lemma \ref{lem comm index} implies that the localised equivariant index of $D$ equals
\beq{eq alt def loc ind}
\ind_G^Z(D) = \ind_G^{Z, L^2(E)}(D) \oplus 0.
\eeq

\begin{example} \label{ex ind non adm}
The importance of using admissible modules in the definition of the (localised) equivariant coarse index is clear in the simplest case, where $G$ is compact, $M$ is a point, $E = V$, an irreducible representation of $G$  (with the trivial grading), and $D = 0_V$ is the zero operator on $V$. We take $Z = M$, so  the localised index equals the non-localised index. We have $L^2(E) = V$, and\footnote{
There is a technical subtlety in the case where $M$ is a finite set: finite-dimensionality of $V$ implies that it is not a standard module (see Definition \ref{def C0X module}). The conditions in Theorem \ref{thm exist adm} are not satisfied, since $G/K$ and $M/G$ are both finite sets. Furthermore, we now have $D^*(M; V)^G = C^*(M; V)^G$ if $D^*(M; V)^G$ is a subalgebra of $\cB(V)^G$. Then $D^*(M; V)^G/C^*(M; V)^G$ is the zero algebra. All of these issues can be solved by tensoring $V$ by $l^2(\N)$, see also Remark \ref{rem H inf dim}.} 
\beq{eq bD cpt}
[b(D)] = [0_V] \in K_1\bigl(D^*(M; V)^G/C^*(M; V)^G\bigr).
\eeq

In that case, Schur's lemma implies that
\[
C^*(M; V)^G = \End(V)^G = \C I_V,
\]
where $I_V$ is the identity operator on $V$. The map $j$ in \eqref{eq def j} is now given by
\[
j(v) = 1\otimes v,
\]
for $v \in V$, where $1$ is the constant function $1$ on $G$. The map
\[
\oplus\,0\colon C^*(M; V)^G =  \C I_V \to \bigoplus_{W \in \hat G} \End(W) = C^*(G)
\]
is given by the inclusion map $\C I_V \hookrightarrow \End(V)$. At the level of $K$-theory, the map $\oplus\,0$ is the map
\[
K_0(C^*(M; V)^G ) = \Z \to R(G) = K_0(C^*(G))
\]
mapping $k \in \Z$ to $k[V] \in R(G)$, the representation ring of $G$. The image under $\ind_G^V$ of the class \eqref{eq bD cpt} is the Fredholm index of $0_V$, which is $[V] \in K_0(\C I_V)$. Under the identification of this $K$-theory group with $\Z$, that class is mapped to $1$. Hence
\[
\ind_G^{\pt, V}(0_V) \oplus 0 = [V] \quad \in R(G).
\]
On the other hand, 
$
\ind_G(0_V \oplus 1)
$ 
is the equivariant Fredholm index of the operator $0_V \oplus 1$, which also equals $[V] \in R(G)$.

This example shows that:
\begin{enumerate}
\item we need to use an admissible module to obtain a single $K$-theory group $K_0(C^*(X)^G) = R(G)$ containing all localised, $G$-equivariant indices on $X$;
\item commutativity of the diagram in Lemma \ref{lem comm index}  means that $\ind_G^{L^2(E)}$, defined via a natural $C_0(X)$-module, but landing in a non-canonical $K$-theory group, determines the localised equivariant index with values in $K_0(C^*(G))$. In this example, $\ind_G^{\pt, V}(0_V)$ is just the integer $1$ when viewed as an element of $\Z = K_0(\C I_V)$, but the representation theoretic information about this index is encoded in the map $\oplus\,0$.
\end{enumerate}


\end{example}



\subsection{Special cases} \label{sec special}
 
 Operators $D$ as in Subsection \ref{sec loc ind ell} occur naturally in at least three settings.
 
\subsubsection*{Callias-type operators}

First of all,  let $\tilde D$ be a Dirac-type operator on $E$. Let $\Phi$ be a $G$-equivariant vector bundle endomorphism of $E$ such that $\tilde D\Phi + \Phi \tilde D$ is a vector bundle endomorphism, and
\beq{eq Callias est}
\tilde D\Phi + \Phi \tilde D + \Phi^2 \geq c^2
\eeq
outside a cocompact subset $Z \subset M$, for a constant $c>0$. 
This can, for example,  be guaranteed by constructing $\Phi$ from projections in the Higson corona algebra as in \cite{Guo18}. (In that case, the pointwise norm of $\tilde D\Phi + \Phi \tilde D$ goes to zero at infinity, while the norm of $\Phi^2$ goes to one.) 
Then the \emph{Callias-type operator}
\beq{eq Callias}
D = D_{\Phi} := \tilde D + \Phi
\eeq
has the properties in Subsection \ref{sec loc ind ell}. 

If $G$ is the trivial group, i.e.\ in the non-equivariant case, index theory of these operators was studied and applied in various places \cite{Bunke95, Callias78, Kucerovsky01}. The coarse geometric viewpoint we develop in this paper could already be useful in the case of trivial groups. Possibly more useful in the non-equivariant setting is to consider the lift of a Callias-type operator $D_{\Phi}$ on a manifold $M$, with fundamental group $\Gamma$, to a $\Gamma$-equivariant Callias-type operator on the universal cover of $M$. The localised equivariant coarse index in $K_*(C^*_{\red}(\Gamma))$ of this lift is a more refined invariant than the Fredholm index of $D_{\Phi}$ itself. 
This could for example yield more refined obstructions to Riemannian metrics of positive scalar curvature.
In future work, we intend to show that the Fredholm index of $D_{\Phi}$ can be recovered from the localised equivariant coarse index of its lift  via an application of the von Neumann trace, or the summation trace in the context of maximal group $C^*$-algebras.

 In \cite{Cecchini16}, Callias-type index theory is extended to operators on bundles of Hilbert modules over $C^*$-algebras $A$, with indices in $K_*(A)$. If $A = C^*_{\red}(G)$ this seems related to the index we study here; in fact we suspect that the two coincide if $G$ is the fundamental group of $M/G$, and $M$ is its universal cover, and the Hilbert module bundle in question is constructed from the natural bundle $M \times_G C_{\red}^*(G) \to M/G$. In the general equivariant case, index theory of Callias-type operators was developed in \cite{Guo18}; we will see in Theorem \ref{thm Callias} that Definition \ref{def loc index} generalises the index of \cite{Guo18}.

 \subsubsection*{Positive curvature at infinity}
 
Secondly,  suppose that $D$ is a Dirac-type operator satisfying a Weitzenb\"ock-type formula
 \[
 D^2 = \nabla^*\nabla + R,
 \]
 for a vector bundle endomorphism $R$ satisfying $R \geq c^2$ outside $Z$. Then $D$ satisfies the conditions in Subsection \ref{sec loc ind ell}, and therefore has a well-defined index in $K_0(C^*_{\red}(G))$. The case where $G$ is trivial (so that $Z$ is compact) was studied by Gromov and Lawson \cite{Gromov83} and applied to questions about Riemannian metrics of positive scalar curvature. The case where $G$ is trivial and $Z$ may be noncompact was treated by Roe \cite{Roe16} using coarse geometry.

\subsubsection*{Manifolds with boundary}

Finally, let $\tilde M$ be a Riemannian manifold with boundary, on which $G$ acts properly, isometrically and cocompactly. Suppose that a neighbourhood $U$ of $\partial \tilde M$ is $G$-equivariantly and isometrically diffeomorphic to a collar $\partial \tilde M \times [0,\varepsilon)$. 
Let $\tilde E \to \tilde M$  be a $G$-equivariant, $\Z_2$-graded, Hermitian vector bundle, and a module over the Clifford bundle of $\tilde M$. Suppose that $\tilde E|_U \cong \tilde E|_{\partial \tilde M} \times [0,\varepsilon)$ as equivariant, Hermitian vector bundles.
Suppose that $\tilde D$ is a Dirac-type operator on $M$, and that on $U$ it is of the form 
\beq{eq D on U}
D|_U = \sigma \circ \bigl(\frac{\partial}{\partial t} + D_{\partial \tilde M}\bigr), 
\eeq
where $\sigma\colon \tilde E_+|_{\partial \tilde M} \to  \tilde E_-|_{\partial \tilde M}$ is an equivariant vector bundle isomorphism, $t$ is the coordinate in $[0,\varepsilon)$, and $D_{\partial \tilde M}$ is a Dirac operator on $\tilde E_+|_{\partial \tilde M}$.

Form $M$ by attaching a cylinder $\partial \tilde M \times [0,\infty)$ to $\tilde M$. Extend the Riemannian metric and the action by $G$ to $M$ in the natural way. Let $E \to M$ be the natural extension of $\tilde E$, and let $D$ be the extension of $\tilde D$ to $E$ equal to \eqref{eq D on U} on $\partial \tilde M \times [0,\infty)$.

Suppose that $D_{\partial \tilde M}$ is invertible. Then $D_{\partial \tilde M}^2 \geq c^2$ for some $c>0$, and $D^2 \geq c^2$ outside $Z = \tilde M$. Hence $D$ satisfies the conditions in Subsection \ref{sec loc ind ell}. The index of Definition \ref{def loc index ell} is now an equivariant Atiyah--Patodi--Singer type index for proper actions, and reduces to the original APS index if $G$ is trivial. An index theorem for this index is proved in \cite{HWW}. As in the case of Callias operators, a special case is the lift of an operator on a compact manifold with boundary to the universal cover, in which case one obtains a refinement of the Atiyah--Patodi--Singer index in the $K_*(C^*_{\red}(\pi))$, where $\pi$ is the fundamental group of the compact manifold.


\section{Results} \label{sec results}


 We will show that the index of Definition \ref{def loc index ell} generalises the equivariant index of Callias-type operators introduced in \cite{Guo18} (Theorem \ref{thm Callias}). As applications, we obtain results on existence and non-existence of Riemannian metrics of positive scalar curvature in Subsection \ref{sec psc}, and discuss a localised version of the Baum--Connes conjecture (Conjecture \ref{conj LBCC}).

\subsection{The equivariant Callias index} \label{sec Callias}

Suppose that $D = D_{\Phi}$ is a Callias-type operator as in \eqref{eq Callias}. 
%
%
Let $\mathcal{E}$ denote the Hilbert $C^*_{\red}(G)$-module defined by completing the space  $\Gamma_c^\infty(E)$ with respect to the $C_c(G)$-valued inner product
$$\langle s,t\rangle(g):=\langle s,gt\rangle_{L^2(E)}$$
and the right action of $C_c(G)$ defined by
$$s\cdot b:=\int_G g^{-1}(b(g)s)\,dg,$$
for $s_1, s_2 \in \Gamma_c^{\infty}(E)$ and $g \in G$. One can find a continuous, $G$-invariant, cocompactly supported function $f$ on $M$ such that $D_{\Phi}^2+f$ is invertible in the sense of the Hilbert $C^*_{\red}(G)$-modules $\mathcal{E}^j$ as in Definition 1 in  \cite{Guo18}. We can then form the normalised $G$-Callias-type operator
\beq{eq F}
F:=D_{\Phi}(D_{\Phi}^2+f)^{-1/2}.
\eeq
Then $F$ lies in the $C^*$-algebra  $\mathcal{L}(\mathcal{E})$ of  bounded adjointable operators  on $\mathcal{E}$. It was shown in Theorem 25 in \cite{Guo18} that $F$ is invertible modulo the algebra $\mathcal{K}(\mathcal{E})$ of compact operators on $\cE$,  and thus defines a class 
\[
[F]\in K_1(\mathcal{L}(\mathcal{E})/\mathcal{K}(\mathcal{E})).
\]
Here, as in \cite{Guo18}, we assume that $E$ is $\Z_2$-graded and $D_{\Phi}$ is odd with respect to the grading, and the above $K$-theory class is defined in terms of the even part of $F$, as in the second point in Definition \ref{def coarse index}.

Let
\beq{eq bdry Callias}
\partial\colon K_1(\mathcal{L}(\mathcal{E})/\mathcal{K}(\mathcal{E})) \to K_0(\mathcal{K}(\mathcal{E})) = K_0(C^*_{\red}(G))
\eeq
be the boundary map associated to the short exact sequence
\[
0 \to \mathcal{K}(\mathcal{E}) \to \mathcal{L}(\mathcal{E}) \to \mathcal{L}(\mathcal{E})/\mathcal{K}(\mathcal{E}) \to 0.
\]
In \eqref{eq bdry Callias}, we have used the Morita equivalence $\mathcal{K}(\mathcal{E}) \sim C^*_{\red}(G)$.

In \cite{Guo18}, the following index was constructed and applied.
\begin{definition} \label{def Callias index}
The \emph{equivariant Callias-index} of $D_{\Phi}$ is
\[
\ind^C_G(D_{\Phi}) := \partial[F] \in K_0(C^*_{\red}(G)).
\]
\end{definition}

One of our main results in this paper is that this index is a special case of the localised equivariant index. This gives a new approach to Callias index theory.
\begin{theorem}\label{thm Callias}
We have
\beq{eq thm Callias}
\ind_G^{\loc}(D_{\Phi}) = \ind_G^C(D_{\Phi}) \quad \in K_0(C^*_{\red}(G)).
\eeq
\end{theorem} 
This result is proved in Section \ref{sec pf Callias}. 

If $M/G$ is compact, then we may take $\Phi = 0$, and  $\ind_G^C$ equals the analytic assembly map \cite{Connes94}.  Therefore, Theorem \ref{thm Callias} has the following immediate consequence.
\begin{corollary} \label{cor loc index ass map}
If $M/G$ is compact, then the localised equivariant index of an elliptic operator is its image under the analytic assembly map.
\end{corollary}
Note that if $M/G$ is compact, then the localised equivariant coarse index equals the usual equivariant coarse index.
In the case of discrete groups, Corollary \ref{cor loc index ass map} is the well-known fact that the equivariant coarse index for a proper action equals the analytic assembly map for such groups \cite{Roe02}.

\subsection{Positive scalar curvature} \label{sec psc}

In the second special case in Subsection \ref{sec special}, if $R$ is uniformly positive, i.e.\ $Z = \emptyset$, then its localised coarse index vanishes by standard arguments. Thus $\ind_G^{\loc}(D) \in K_*(C^*_{\red}(G))$ is an obstruction to $G$-invariant Riemannian metrics of positive scalar curvature. There are many techniques for extracting more concrete, numerical obstructions from this $K$-theory class, such as pairing with traces and higher cyclic cocycles on (smooth subalgebras of) $C^*_{\red}(G)$.

In the case where $Z/G$ is non-compact, the localised equivariant coarse index allows us to use
the following method to find obstructions to $G$-invariant Riemannian metrics of positive scalar curvature. This generalises the comments at the start of Section 3 in \cite{Roe16}.
\begin{proposition} \label{prop psc}
Let $M$ be a complete Riemannian manifold, with a proper, isometric action by a locally compact group $G$. 
 Let $D$ be a Dirac-type operator whose curvature term $R$ in the Weitzenb\"ock-type formula $D^2 = \Delta + R$ is uniformly positive outside a $G$-invariant subset $Z \subset M$, for which the inclusion map $C^*(M; Z)^G\to C^*(M)^G$ induces the zero map on $K$-theory, with respect to an admissible $C_0(M)$-module $H_M$ and its restriction $H_Z:=\one_Z H_M$.  Then $\ind_G(D)=0$. 
\end{proposition}
When $Z$ is cocompact, it is clear that the inclusion
\[
K_*(C^*(M; Z,H_M)^G)= K_*(C^*(Z,\one_Z H_M)^G).
\]
induces the identity map on $K$ theory. More generally, we expect that, as in the discrete group case,\footnote{See the forthcoming book \cite{WillettYu}.} the equivariant Roe algebras of coarsely equivalent spaces have canonically isomorphic $K$-theory, and hence that this identity holds for general $Z$ (see e.g.\ Lemma 1 in Section 5 of \cite{HRY93}, or Proposition 6.4.7 in \cite{Higson00} for the non-equivariant case.) Then the condition on the set $Z$ in the above proposition is satisfied for example if $Z$ is contained in a subset $Y \subset M$ such that $K_*(C^*(Y)^G)=0$. In future work, we aim to prove index theorems that allow us to deduce concrete topological obstructions to positive scalar curvature from Proposition \ref{prop psc}.

We now turn to an existence result. Recall the following theorem  from \cite{Abels}, which we need only for Lie groups:
	\begin{theorem}[Abels] If $M$ is a proper $G$-manifold, where $G$ is an almost connected Lie group, 
	then there exists a global slice $N$ which is a $K$-manifold, in the $G$-manifold $M$, where $ K$ is a maximal compact subgroup of $G$. 
	\end{theorem}
By this theorem, $M$ is $G$-equivariantly diffeomorphic to $G\times_K N$.
	\begin{theorem}\label{thm:induction}
		Let $G$ be an almost connected Lie group, and let $K$ be a maximal compact subgroup of $G$. If $N$ is a bounded geometry manifold with a $K$-invariant Riemannian metric of uniform positive scalar curvature, then  $M=G\times_K N$ is a bounded geometry manifold with a $G$-invariant Riemannian metric of  uniform positive scalar curvature.
	\end{theorem}

\subsection{A localised Baum--Connes conjecture}

The Baum--Connes conjecture \cite{Connes94} describes $K_*(C^*_{\red}(G))$ in terms of equivariant indices of elliptic operators for cocompact actions by $G$. The surjectivity part of this conjecture  is a particularly hard problem. Using the localised equivariant index of Definition \ref{def loc index}, we will formulate a localised version of Baum--Connes surjectivity. We show that this is implied by Baum--Connes surjectivity in the usual sense (Proposition \ref{prop LBCC} below). It is therefore a weaker statement - and potentially easier to prove because one is allowed to use equivariant indices for \emph{non-cocompact} actions - but which nevertheless describes the group $K_*(C^*_{\red}(G))$.

Let $D^*_{\red}(X)^G$ be the algebra defined in Subsection \ref{sec loc ind}.
\begin{definition}
The \emph{localised equivariant $K$-homology} of $X$ is
\[
K_*^G(X)_{\loc} := K_{*+1}(D^*_{\red}(X)^G/C^*(X)^G_{\loc}).
\]
\end{definition}
This terminology is motivated by Paschke duality (see e.g.\ page 85 of \cite{Roe03} and Theorem 8.4.3 in \cite{Higson00}), which implies that $K_*^G(X)_{\loc}$ equals the usual equivariant $K$-homology of $X$ in the opposite degree, if $X/G$ is compact. 
The index of Definition \ref{def loc index} defines a localised equivariant index
\[
\ind_G^{\loc} \colon K_*^G(X)_{\loc} \to K_*(C^*_{\red}(G)).
\]

Now let $\underline{E}G$ be a universal example for proper $G$-actions \cite{Connes94}. 
\begin{conjecture}[Localised Baum--Connes surjectivity]\label{conj LBCC}
The map
\[
\ind_G^{\loc}\colon K_*^G(\underline{E}G)_{\loc}\to K_*(C^*_{\red}(G))
\]
is surjective.
\end{conjecture}

Recall that the \emph{representable equivariant $K$-homology} of $X$ is
\[
RK^G_*(X) := \varinjlim_{Z} K^*_G(Z),
\]
where $Z$ runs over the $G$-invariant closed subsets of $X$ such that $Z/G$ is compact. The {Baum--Connes conjecture}  is the statement that the \emph{analytic assembly map}
\[
\mu_G\colon
RK^G_*(\underline{E}G)\to K_*(C^*_{\red}(G))
\]
is bijective.
\begin{proposition}\label{prop LBCC}
Surjectivity of $\mu_G$ implies Conjecture \ref{conj LBCC}.
\end{proposition}

The converse of Proposition \ref{prop LBCC} is directly related to the question of whether the equivariant localised index of Definitions \ref{def loc index} and \ref{def loc index ell} lands in the image of the Baum--Connes assembly map. This question was posed for the equivariant Callias index in \cite{Guo18}, and is open in general. See also Remark \ref{rem BC inj surj} below.

\section{Proofs of properties of equivariant Roe algebras} \label{sec proofs Roe}

In Subsections \ref{sec exist adm}--\ref{sec pf thm exist}, we prove Theorem \ref{thm exist adm}, which guarantees the existence of geometric admissible modules. We then use this in Subsection \ref{sec kernels} to prove Theorem \ref{thm Roe group Cstar}.

\subsection{An isomorphism of $G$-representations} \label{sec exist adm}

Since $G$ is a non-compact group, we decompose the space $X$ using Palais' local slice theorem \cite{Palais61}. Let $K_x$ denote the stabiliser of a point $x\in X$. This result states that for every point $x \in X$,   there is a $K_x$-invariant subset $Y_x \subset X$ such that the action map $G \times Y_x \to X$ descends to a $G$-equivariant homeomorphism from $G \times_{K_x} Y_x$ onto a $G$-invariant  open neighbourhood $W_x$ of $x$. Note that the stabiliser $K_x$ is compact, since the action is proper. 

Consider a covering of $X$ by sets of the form $G \times_{K_j} Y_j$, for compact subgroups $K_j < G$ and compact, $K_j$-invariant slices $Y_j \subset X$. Suppose that the intersections between these sets have measure zero. For each $j$, fix a $K$-invariant measure $dy_j$ on $Y_j$. Together with the Haar measure $dg$, they induce a $G$-invariant measure $dx$ on $X$. We will call such a measure \emph{induced from slices}. Such measures are natural choices; see for example Lemma 4.1 in \cite{HSIII}.

We will use the following fact, whose proof is straightforward.
\begin{lemma}[Fell absorption]\label{lem Fell}
If $\pi\colon G \to \U(\HH)$ is a unitary representation, and $\lambda\colon G \to \U(L^2(G))$ is the left-regular representation, then the map
\[
\Phi\colon
L^2(G)\otimes \HH \to L^2(G)\otimes \HH,
\] 
defined by
\[
\Phi(f)(g) = \pi(g) f(g),
\]
for $f \in L^2(G, \HH)$ and $g \in G$, is a unitary isomorphism intertwining the representations $\lambda \otimes \pi$ and $\lambda \otimes 1$.
\end{lemma}
Suppose first that $X = G\times_K Y$, for a $K$-space $Y$. Let $dx$ be a measure
on $X$ be induced from the measure $dg$ on $G$ and a $K$-invariant measure $dy$ on $Y$. Consider the measure $d(Kg)$ on $K \backslash G$.
Choose a measurable section $\phi\colon K \backslash G\rightarrow G$. 
\begin{lemma}\label{lem measurable iso}
	The map
	\[ 
	\psi\colon X\times G=(G\times_K Y)\times G\cong G\times K\backslash G\times Y,
	\]
	given by
	\[
	\psi
	([g,y],h)=(h(\phi(Kg^{-1}h)^{-1}),Kg^{-1}h,(\phi(Kg^{-1}h)h^{-1}g)y)
	\]
	is a 
	$G$-equivariant, measurable bijection. It relates the measures $dx \times dg$ and $dg \times d(Kg) \times dy$ to each other.
\end{lemma}
\begin{proof}
	There is a $K$-equivariant isomorphism of measure spaces (by which we mean a measurable bijection relating the given measures on the two spaces)
	$$G\rightarrow K\times(K\backslash G),\qquad g\mapsto(g(\phi(Kg)^{-1}),Kg).$$
	Thus we have a $G$-equivariant isomorphism of measure spaces
	$$G\times_K(G\times Y)\cong G\times_K(K\times K\backslash G\times Y),$$
	$$[(g,(h,y))]\mapsto[(g,h(\phi(Kh)^{-1}),Kh,y)].$$
	Combining this with the $K$-equivariant isomorphism
	$$K\times Y\rightarrow K\times Y,\qquad (k,y)\mapsto(k,k^{-1}y)$$
	(where $K$ acts diagonally on the left and only on the first factor on the right), this gives a $G$-equivariant isomorphism of measure spaces
	$$G\times_K(G\times Y)\cong G\times_K(K\times K\backslash G\times Y),$$
	$$[(g,(h,y))]\mapsto[(g,h(\phi(Kh)^{-1}),Kh,(\phi(Kh)h^{-1})y)],$$
	where $K$ now acts trivially on $Y$ on the right. Using the identification $G\times_K K\cong G$, $[(g,k)]\mapsto gk$, we get the $G$-equivariant isomorphism
	$$G\times_K(G\times Y)\cong G\times K\backslash G\times Y,$$
	$$[(g,(h,y))]\mapsto (gh(\phi(Kh)^{-1}),Kh,(\phi(Kh)h^{-1})y).$$
	Here $G$ acts only on the first factor of both sides. Note that
	$$(G\times_K Y)\times G\cong G\times_K(G\times Y),$$
	$$\qquad([(g,y)],h)\mapsto[(g,(g^{-1}h,y))]$$
	(where $G$ acts diagonally on the left and only on the first factor on the right). The first claim then follows.
	%
\end{proof}




Set $H := L^2(K \backslash G) \otimes L^2(Y)$.
Pulling back functions along the map $\psi$ in Lemma \ref{lem measurable iso} induces a unitary, $G$-equivariant isomorphism
\[
\psi^*\colon L^2(G) \otimes H \to L^2(X) \otimes L^2(G).
\]

Suppose that $X/G$ is compact. Then, in general, $X$ is a finite union of sets of the form $G \times_{K_j} Y_j$ for compact subgroups $K_j < G$ and compact subsets $Y_j$. These can be chosen so that the overlaps between these sets have measure zero. Lemma \ref{lem measurable iso} yields isomorphisms
\[
\psi_j^*\colon L^2(G) \otimes H_j \to L^2(G \times_{K_j} Y_j) \otimes L^2(G),
\]
where $H_j = L^2(K_j \backslash G) \otimes L^2(Y_j)$. They  combine into a global isomorphism
\beq{eq def Psi}
\Psi\colon L^2(G) \otimes H \xrightarrow{\cong} \bigoplus_j  L^2(G \times_{K_j} Y_j)  \otimes L^2(G) \cong L^2(X) \otimes L^2(G),
\eeq
where $H = \bigoplus_j H_j$. We have proved the following:\footnote{
Here we have taken the bundle $E\rightarrow X$ to be the trivial line bundle. The general case can be proved completely analogously.}

\begin{proposition}\label{prop G rep}
There exists a $G$-equivariant unitary isomorphism $\Psi:L^2(G)\otimes H\rightarrow L^2(X)\otimes L^2(G)$ for a separable Hilbert space $H$. The space $H$ is infinite-dimensional if $G/K$ or $X/G$ is infinite.
\end{proposition}

\subsection{Propagation in $X$ and in $G$}

Next, we show that $\Psi$ coarsely relates propagation on $L^2(G)\otimes H$ with respect to $C_0(G)$ to  propagation on $L^2(X)\otimes L^2(G)$ with respect to $C_0(X)$.
\begin{proposition}\label{prop prop G X}
	Let $\Psi$ be the isomorphism in Proposition \ref{prop G rep}. An operator $T$ on $L^2(X) \otimes L^2(G)$ has finite propagation in $X$ if and only if
	$\Psi^{-1} \circ T \circ \Psi$ has finite propagation in $G$.
\end{proposition}
To show this, we suppose first that $X$ consists of just one slice. That is, $X = G \times_K Y$, for a compact $K$-space $Y$. We will reduce the general case to this case.
Let $\diam(K)$ be the diameter of $K$.
%
\begin{lemma} \label{lem psi1}
	Let 
	\[
	\psi\colon G \times_K Y \times G \to G \times K \backslash G \times Y
	\]
	be the bijective map from Lemma \ref{lem measurable iso}. Let $\psi_1$ be its first component, mapping into $G$. Then for all $g,g', h,h' \in G$ and $y,y' \in Y$,
	\[
	d_G(g,g') - 2\diam(K) \leq d_G(\psi_1([g,y], h), \psi_1([g',y'], h')) \leq d_G(g,g') + 2\diam(K).
	\]
\end{lemma}
\begin{proof}
	If $g, h \in G$ and $y \in Y$, then
	\[
	\psi_1([g,y], h) = h \phi(Kg^{-1}h)^{-1},
	\]
	where $\phi\colon K\backslash G \to G$ is a section. This means that there is a $k \in K$ such that $\phi(Kg^{-1}h) = kg^{-1}h$. Hence
	\[
	\psi_1([g,y], h) = gk^{-1}.
	\]
	Let $g,g', h,h' \in G$ and $y,y' \in Y$. Then we have just seen that there are $k,k' \in K$ such that
	\[
	d_G(\psi_1([g,y], h), \psi_1([g',y'], h')) = d_G(gk^{-1}, g'k'^{-1}).
	\]
	This lies in the range specified by the triangle inequality and  left invariance of $d$.
\end{proof}
\begin{lemma} \label{lem dG dX}
	For all $s>0$ there are $r, r'>0$ such that for all $g,g' \in G$ and $y,y' \in Y$, 
	\[
	\begin{split}
	d_G(g,g') \leq s \quad &\Rightarrow \quad d(gy, g'y') \leq r. \\
	d(gy, g'y')  \leq s  \quad &\Rightarrow \quad d_G(g,g')  \leq r'.
	\end{split}
	\]
\end{lemma}
\begin{proof}
	Let $s>0$ be given. Define
	\[
	r:= \max \{d(gy, y'); y,y' \in Y, g \in G, d_G(g,e) \leq s\}.
	\]
	Here we use compactness of $Y$.
	Then for all $g,g' \in G$ with $d_G(g,g') \leq s$, we have $d_G(g'^{-1}g, e) \leq s$, so for all $y,y' \in Y$,
	\[
	d(gy, g'y') = d(g'^{-1}gy, y') \leq r.
	\]
	
	To prove the second claim, note that properness of the action by $G$ on $X$ and compactness of $Y$ imply that the set
	\[
	A_s := \{g \in G; gY \cap \Pen(Y, s) \not= \emptyset \}
	\]
	is compact (with notation as in \eqref{eq def Pen}). Set
	\[
	r' := \max\{d_G(g,e); g \in A_s\}.
	\]
	Then for all $g,g' \in G$ and $y,y' \in Y$ with $d(gy, gy') \leq s$, we have $g'^{-1}g \in A_s$, so
	$
	d_G(g, g') \leq r'
	$.
\end{proof}
\begin{lemma} \label{lem prop X prop G}
	If an operator $T$ on $L^2(X) \otimes L^2(G)$ has finite propagation in $X$, then
	$\Psi^{-1} \circ T \circ \Psi$ has finite propagation in $G$.
\end{lemma}
\begin{proof}
	Suppose that $T$ is an operator on $L^2(X) \otimes L^2(G)$ with finite propagation $s$ in $X$. By the second part of Lemma \ref{lem dG dX}, there is an $r>0$ such that for all $g,g' \in G$ and $y,y' \in Y$, 
	\[
	d_G(g,g')  \geq r \Rightarrow
	d(gy, g'y')  \geq s+1.
	\]
	Let $\chi_1, \chi_2 \in C_c(G)$ be given, with $d_G(\supp(\chi_1), \supp(\chi_2)) \geq r+2\diam(K)$.
	For $j = 1,2$, let $g_j \in \supp(\chi_j)$, $h_j \in G$ and $y_j \in Y$ be given. Write
	\[
	(\tilde g_j \tilde y_j, \tilde h_j) = \psi^{-1}(g_j, Kh_j, y_j),
	\]
	for $\tilde g_j, \tilde h_j \in G$ and $\tilde y_j \in Y$. Then by Lemma \ref{lem psi1},
	\[
	d_G(\tilde g_1, \tilde g_2) \geq d_G(g_1, g_2) - 2\diam(K) \geq r.
	\]
	So $d(\tilde g_1 \tilde y_1, \tilde g_2 \tilde y_2) \geq s+1$. Let $\pi_X\colon X\times G \to X$ be the projection onto the first factor. We have just seen that
	\[
	d\bigl(\pi_X(\psi^{-1}(\supp \chi_1 \times K \backslash G \times G)), \pi_X(\psi^{-1}(\supp \chi_2 \times K \backslash G \times G))\bigr) \geq s+1.
	\]
	Hence we can choose $\varphi_j \in C_c (X)$, for $j=1,2$, such that $\varphi_j\equiv 1$ on $\pi_X(\psi^{-1}(\supp \chi_j \times K \backslash G \times G))$, and
	\[
	d(\supp(\varphi_1), \supp(\varphi_2)) \geq s.
	\]
	We conclude that
	\[
	\chi_1(\Psi^{-1}\circ T \circ \Psi) \chi_2 = 
	\Psi^{-1}\circ (\psi^*(\chi_1 \otimes 1_{K \backslash G \times Y})\varphi_1 T \varphi_2  \psi^*(\chi_2 \otimes 1_{K \backslash G \times Y}))\circ \Psi = 0,
	\]
	since $\varphi_1 T \varphi_2 =0$.
\end{proof}
\begin{lemma} \label{lem prop G prop X}
	If an operator $\tilde T$ on $L^2(G) \otimes H$ has finite propagation in $G$, then 
	$\Psi \circ \tilde T \circ \Psi^{-1}$  has finite propagation in $X$.
\end{lemma}
\begin{proof}
	Let $\tilde T$ be an operator on $L^2(G) \otimes H$ with finite propagation $s$ in $G$. The first part of Lemma \ref{lem dG dX} implies that there is an $r>0$ such that for all $g,g' \in G$ and $y,y' \in Y$,
	\beq{eq dist gy g}
	d(gy, g'y') \geq r \quad \Rightarrow \quad d_G(g,g') \geq s+1 +2\diam(K).
	\eeq
	For $j=1,2$, let $\varphi_j \in C_c(X)$ be such that $d(\supp(\varphi_1), \varphi_2) \geq r$. 
	Let $\pi_G\colon G \times K\backslash G \times Y \to G$ be the projection onto the first factor.
	By Lemma \ref{lem psi1}, we have, for all $g_j, h_j \in G$ and $y_j \in Y$ such that $g_j y_j \in \supp(\varphi_j)$,
	\[
	d_G\bigl(\pi_G(\psi(g_1y_1, h_1)), \pi_G(\psi(g_2y_2, h_2)) \bigr) \geq d_G(g_1, g_2)-2\diam(K) \geq s+1,
	\]
	where we have used \eqref{eq dist gy g}. So we can choose $\chi_j \in C_c(G)$ such that $\chi_j\equiv 1$ on $\pi_G(\psi(\supp(\varphi_j)\times G))$ and $d_G(\supp(\chi_1), \supp(\chi_2)) \geq s$. Then
	\[
	\varphi_1 (\Psi \circ \tilde T \circ \Psi^{-1}) \varphi_2 = 
	\Psi \circ ((\psi^{-1})^*(\varphi_1 \otimes 1_G)  \chi_1 \tilde T  \chi_2  (\psi^{-1})^*(\varphi_2 \otimes 1_G)   ) \circ \Psi^{-1} =0,
	\]
	since $\chi_1 \tilde T  \chi_2 =0$.
\end{proof}
\begin{proof}[Proof of Proposition \ref{prop prop G X}]
	If $X = G\times_K Y$ for a single, compact slice $Y \subset X$, then the claim is precisely Lemmas \ref{lem prop X prop G} and \ref{lem prop G prop X}.
	
	In the general case when the cocompact space $X$ consists of finitely many slices, $L^2(X)\otimes L^2(G)$ is a finite direct sum $\bigoplus_i L^2(G)\otimes H_i$ as in \eqref{eq def Psi}. 
	An operator $T$ on this space can be written as a finite matrix $(T_{i,j})$, where each entry $T_{i,j}$ is an operator
	$$T_{i,j}:L^2(G)\otimes H_i\rightarrow L^2(G)\otimes H_j.$$
	The result then follows from the case of a single slice.
\end{proof}

\subsection{Proof of Theorem \ref{thm exist adm}} \label{sec pf thm exist}

The remaining step in the proof of Theorem \ref{thm exist adm} is to show that $\Psi$  relates local compactness of operators on $H_X$ with respect to $C_0(X)$ to local compactness of operators on $L^2(G)\otimes H$ with respect to $C_0(G)$.
\begin{proposition}
\label{prop loc cpt X G}
A bounded operator $T$ on $H_X$ is locally compact with respect to the action by $C_0(X)$ if and only if the bounded operator $\Psi^{-1}\circ T \circ \Psi$ on $L^2(G)\otimes H$ is locally compact with respect to the action by $C_0(G)$.
\end{proposition}
\begin{proof}
First, suppose that $X = G \times_K Y$ for a single slice $Y$.

Suppose $T$ is a bounded operator on $L^2(X)\otimes L^2(G)$ that is locally compact with respect to multiplication by $C_0(X)$. Let $\chi\in C_c(G)$ be given. 
As in the proof of Lemma \ref{lem prop X prop G}, Lemmas \ref{lem psi1} and \ref{lem dG dX} imply that the subset $\pi_X(\psi^{-1}(\supp\chi\times K\backslash G\times G))$ of $X$ is bounded. Hence we can choose $\phi\in C_c(X)$ such that $\phi\equiv 1$ on $\pi_X(\psi^{-1}(\supp\chi\times K\backslash G\times G))$. Thus
$$(\Psi^{-1}\circ T\circ\Psi)\chi=\Psi^{-1}\circ(T\phi\psi^*(\chi\otimes1_{K\backslash G\times Y}))\circ\Psi\in\mathcal{K}(L^2(G)\otimes H),$$
since $T\phi\in\mathcal{K}(L^2(X)\otimes L^2(G))$.

Now suppose $\tilde{T}=\Psi^{-1}\circ T\circ\Psi$ is a bounded operator on $L^2(G)\otimes H$ that is locally compact with respect to the multiplicative action of $C_0(G)$ on the first factor. Let $\phi\in C_c(X)$ be given. As in the proof of Lemma \ref{lem prop G prop X}, Lemmas \ref{lem psi1} and \ref{lem dG dX} imply that the subset  $\pi_G(\psi(\supp(\phi)\times G))$ of $G$ is bounded. Hence
 we can choose $\chi\in C_c(G)$ such that $\chi\equiv 1$ on $\pi_G(\psi(\supp(\phi)\times G))$. Then
$$T\phi=(\Psi\circ\tilde{T}\circ\Psi^{-1})\phi=\Psi\circ(\tilde{T}\chi(\psi^{-1})^*(\phi\otimes 1_G))\circ\Psi^{-1}\in\mathcal{K}(L^2(X)\otimes L^2(G)),$$
since $\tilde{T}\chi\in\mathcal{K}(L^2(G)\otimes H).$

The general case for cocompact $X$ and more than one slice follows from the single slice case as in the proof of Proposition \ref{prop prop G X}.
\end{proof}

Theorem \ref{thm exist adm} is the combination of Propositions \ref{prop prop G X} and \ref{prop loc cpt X G}. Under the assumption in Theorem \ref{thm exist adm} that $G/K$ or $X/G$ is infinite, the space $H$ in Proposition \ref{prop G rep} is infinite-dimensional.

\subsection{Kernels and group $C^*$-algebras} \label{sec kernels}

The reduced and maximal equivariant Roe algebras can alternatively be described in terms of  continuous Schwartz kernels of operators. We work this out in detail this subsection and use it to prove Theorem \ref{thm Roe group Cstar}.

Let $X$ and $G$ be as before. Suppose that $X/G$ is compact. Let $H_X$ be any admissible equivariant $C_0(X)$-module over $X$. In this subsection, we will always identify $H_X$ with $L^2(G) \otimes H$ for an infinite-dimensional separable Hilbert space $H$, using the isomorphism $\Psi$ in Definition \ref{def adm mod}.


%
%
\begin{definition} \label{def:admissiblekernelalgebra}
	Let $C^*_{\ker}(X)^G$ denote the algebra of bounded operators on $H_X$ defined by continuous $G$-invariant Schwartz kernels
	$$\kappa \colon G\times G\rightarrow\mathcal{K}(H)$$
	that have finite propagation in $G$. 
\end{definition}
By the `if' part of Proposition \ref{prop prop G X}, we have
$C^*_{\ker}(X)^G \subset C^*_{\alg}(X)^G$.
Our goal is to prove the following proposition.
\begin{proposition}\label{prop dense kernels}
$C^*_{\ker}(X)^G$ is dense in $C^*_{\alg}(X)^G$ with respect to the operator norm on $H_X$.
\end{proposition}
Because of this proposition, the Roe algebra $C^*(X)^G$ can alternatively be defined as the closure of $C^*_{\ker}(X)^G$ in $\cB(H_X)$.
This immediately implies Theorem \ref{thm Roe group Cstar}, since 
\[
C^*_{\ker}(X)^G \cong C_c(G) \otimes \cK(H)
\]
via the isomorphism sending $\kappa \in C^*_{\ker}(X)^G$ to the map $g\mapsto \kappa(g^{-1},e)$.

Let $T\in C^*_{G,\alg}(X)$. We will prove Proposition \ref{prop dense kernels} by showing that $T$ can be approximated by elements of $C^*_{\ker}(X)^G$ in the operator norm.

Fix ${\chi}\in C_c^\infty(G)$ such that 
for all $h \in G$,
$$\int_G{\chi}(g^{-1}h)\,dg=1.$$
Let $H$ be as in Definition \ref{def adm mod}. 
By the `only if' part of Proposition \ref{prop prop G X}, 
 $T$ has finite propagation in $G$. So there exist functions ${\chi_1},\chi_2\in C_c(G)$ such that 
\[
\begin{split}
\chi_1T\chi&=T\chi,\\
\chi \chi_2 &= \chi
\end{split}
\]
when $T$ is viewed as an operator on $L^2(G)\otimes H$.

Let $\{e_j\}_{j=1}^\infty$ be a Hilbert basis for $H_X \cong L^2(G, H)$ such that $e_j\in C_c(G,H)$ for every $j$. Let $\{e^k\}_{k=1}^\infty$ be the dual basis. We view $e^k$ as the element of $C_c(G, H^*)$ such that for all $g \in G$ and $v \in H$,
\[
e^k(g)(v) = (e_k(g), v)_H.
\]
By the definition of $C^*_{G,\textnormal{alg}}(X)$ and Proposition \ref{prop loc cpt X G}, we have
$T\chi\in\mathcal{K}(H_X)$. Thus we can write
\begin{equation} \label{eq sum Tchi}
	T\chi=\sum_{j,k}a_k^j e_j\otimes e^k
\end{equation}
for some constants $a_k^j$, with the sum converging in operator norm in $\mathcal{B}(L^2(G)\otimes H)$. Define $T^j_k\in\mathcal{B}(L^2(G)\otimes H)$ to be the operator given by the Schwartz kernel
$$\kappa^j_k:G\times G\rightarrow\mathcal{K}(H),$$
\beq{eq def Kjk}
(h,h')\mapsto a_k^j\int_G\chi_1(g^{-1}h) \chi_2(g^{-1}h') e_j(g^{-1}h) \otimes e^k(g^{-1}h') \,dg,
\eeq
where $h,h'\in G$. 
Since $e_j(g) \otimes e^k(g')$ (for  $g,g' \in G$) is a finite-rank operator on $H$ and the integrand in \eqref{eq def Kjk} is compactly supported, we find that indeed $\kappa^j_k(h,h') \in \cK(H)$ for all $h,h' \in G$. Furthermore, 
 $\kappa_k^j$ is continuous, $G$-invariant, and has finite propagation in $G$. 
\begin{lemma}\label{lem sum conv}
	For every $f \in L^2(G,H)$ and $h \in G$,
	\[
	(Tf)(h)=\sum_{j,k=1}^\infty T_k^j f(h).
	\]
	%
\end{lemma}
\begin{proof}
	Let $f\in L^2(G,H)$. Then for every $g\in G$ we have
	\begin{align*}
	T\circ(g\cdot\chi)&=g(T\chi)g^{-1}\\
	&=g\chi_1T\chi  \chi_2 g^{-1}.
	\end{align*}
	Thus for all $h\in G$,
	\begin{align*}
	(Tf)(h)&=\int_G\left(T(g\cdot\chi)f\right)(h)\,dg\\
	&=\int_G\left((g\chi_1T\chi \chi_2 g^{-1})f\right)(h)\,dg\\
	&=\sum_{j,k}a_k^j\int_G\chi_1(g^{-1}h)e_j(g^{-1}h)\left(\int_G\bigl(e_k(l), \chi_2(l)f(gl) \bigr)_{H}\,dl\right)\,dg\\
	&=\sum_{j,k}\int_G \kappa_k^j(h,m)f(m)\,dm,
	\end{align*}
	where we substitute $m =gl$. Note that all integrands are continuous and compactly supported, so we may indeed interchange integrals and sums.
\end{proof}
\begin{lemma} \label{lem op norm}
	The sum 
	\[
	\sum_{j,k=1}^{\infty} T_k^j
	\]
	converges in $\mathcal{B}(L^2(G)\otimes H)$ with respect to the operator norm.
\end{lemma}
\begin{proof}
	We have for all $j,k\in \N$ and $f\in C_c(G,H)$,
	\begin{align}\label{eq Tjkf}
	T^{j}_kf &= a^j_k \int_G (g\cdot (\chi_1 e_j)) (g \cdot(\chi_2  e_k), f)_{L^2(G,H)}\, dg\nonumber\\
	&= a^j_k \int_G (g \cdot \chi_1) (g\circ( e_j \otimes e^k) \circ g^{-1}) (g\cdot \chi_2)f\, dg.
	\end{align}
	
	Hence for all $M,N, M', N' \in \N$ with $M \leq M'$ and $N \leq N'$,
	\[
	\begin{split}
	\Bigl\| \sum_{j=M}^{M'} \sum_{k=N}^{N'} T^j_k f \Bigr\|_{L^2(G,H)} &= 
	\Bigl\| \sum_{j=M}^{M'} \sum_{k=N}^{N'} a^j_k \int_{G}  (g \cdot \chi_1) (g\circ( e_j \otimes e^k) \circ g^{-1}) (g\cdot \chi_2)f \, dg \Bigr\|_{L^2(G, H)} \\
	&= 
	\Bigl\|  \int_{G} (g\cdot \chi_1)  \Bigl(g \circ \Bigl(\sum_{j=M}^{M'} \sum_{k=N}^{N'} a^j_k   e_j  \otimes  e^k \Bigr) \circ g^{-1}\Bigr) (g\cdot \chi_2)   f\, dg \Bigr\|_{L^2(G,H)}. 
	\end{split}
	\]
	Write
	\[
	T_{M,N}^{M',N'}  :=  \sum_{j=M}^{M'} \sum_{k=N}^{N'} a^j_k  e_j  \otimes  e^k.
	\]
	Define $F \colon G \to L^2(G,H)$ by
	\[
	F(g) = (g\cdot \chi_1) (g\circ T_{M,N}^{M',N'} \circ g^{-1}) (g\cdot \chi_2) f,
	\]
	for $g \in G$. If $g,g' \in G$, and $(F(g), F(g'))_{L^2(G,H)} \not =0$, then 
	\[
	g\supp(\chi_1) \cap g'\supp(\chi_1) \not= \emptyset.
	\]
	By properness of the action, this means that $g^{-1}g'$ lies in a compact set $S \subset G$, only depending on $\chi_1$.
	%
	%
	By Lemma 1.5 in \cite{Connes82}, this implies that
	\[
	\Bigl\|  \int_{G}F(g) \, dg \Bigr\|_{L^2(G,H)}^2 \leq \vol(S)  \int_{G}\| F(g)\|_{L^2(G,H)}^2 \, dg. 
	\]
	Hence, since $G$ acts unitarily on $L^2(G,H)$,
	\[
	\begin{split}
	\Bigl\| \sum_{j=M}^{M'} \sum_{k=N}^{N'} T^j_k f \Bigr\|_{L^2(G,H)}^2 &
	\leq \vol(S)  \|\chi_1\|_{\infty}^2 \|T_{M,N}^{M',N'}\|_{\cB(L^2(G,H))}^2 \int_G \|(g\cdot \chi_2)f\|_{L^2(G, H)}^2 \, dg \\
	&\leq \vol(S)  \|\chi_1\|_{\infty}^2 \|T_{M,N}^{M',N'}\|_{\cB(L^2(G,H))}^2  \|\chi_2 \|_G^2    \|f\|_{L^2(G,H)}^2,
	\end{split}
	\]
	where
	\[
	\|\chi_1 \|_\infty := \max_{g \in G}||\chi_1(g)||_{\mathcal{B}(H)},
	\]
	\[
	\|\chi_2 \|_G := \sqrt{\max_{h \in G} \int_G||\chi_2(g^{-1}h)||_{\mathcal{B}(H)}^2\, dg}.
	\]
	We conclude that the operator
	\[
	\sum_{j=M}^{M'} \sum_{k=N}^{N'} T^j_k
	\]
	on $L^2(G,H)$ is bounded, with norm at most 
	\[
	\vol(S)^{1/2}  \|\chi_1\|_{\infty}    \|\chi_2 \|_{G} \|T_{M,N}^{M',N'}\|_{\cB(L^2(G,H))}.
	\]
		Since 
	the  sum \eqref{eq sum Tchi} converges in the operator norm and $\cB(L^2(G,H))$ is complete, the claim follows.
\end{proof}

\begin{proof}[Proof of Proposition \ref{prop dense kernels}]
By Lemmas \ref{lem sum conv} and \ref{lem op norm}, we have
	\[
	T = \sum_{j,k=1}^{\infty}T^j_k,
	\]
	where the sum converges in the operator norm. Hence $C^*_{G,\ker}(X)$ is dense in $C^*_{G,\alg}(X)$.
\end{proof}

\section{The equivariant Callias index} \label{sec pf Callias}

In this section, we prove Theorem \ref{thm Callias}, showing that the equivariant index of a $G$-Callias-type operator, as defined in \cite{Guo18}, identifies naturally with its localised equivariant index given by Definition \ref{def loc index ell}. We begin in Subsection \ref{sec ind maps cocpt} by relating the equivariant coarse index defined in Subsection \ref{sec loc ind} for cocompact actions to the usual $G$-equivariant index obtained through the assembly map, before relating the localised equivariant index to the non-cocompact $G$-equivariant index in Subsection \ref{sec ind maps noncocpt}. With respect to the notation in Subsection \ref{sec loc ind}, we are working with $D^*(X)^G=\mathcal{M}(C^*(X)^G)$ or $\mathcal{M}(C^*(X)^G_{\loc})$, depending on context. 

The results in the first two subsections of this section are of a general nature and apply to both the maximal and reduced versions of the index, and we will use $C^*(G)$ will denote either $C^*_{\red}(G)$ or $C^*_{\max}(G)$, and $C^*(X)^G$ (resp. $C^*(X)^G_{\loc}$) for either the reduced or maximal version of the Roe algebra (resp. localised Roe algebra).

\subsection{Index maps in the cocompact case} \label{sec ind maps cocpt}
Suppose that $X$ is $G$-cocompact. 
%
%
%
Equip the dense subspace $C_c(G,L^2(E))$ of $L^2(E)\otimes L^2(G)$ with the $C_c(G)$-valued inner product
$$\langle s,t\rangle(g):=\langle s,gt\rangle_{L^2(E)\otimes L^2(G)}$$
and the right action of $C_c(G)$ defined by
$$s\cdot b:=\int_G g^{-1}(b(g)s)\,dg.$$
Taking the completion gives rise to a Hilbert $C^*(G)$-module $\mathcal{E}_{C^*(G)}.$
\begin{lemma}
\label{lem:standardiso}
$\mathcal{E}_{C^*(G)}$ is isomorphic to the standard Hilbert $C^*(G)$-module $C^*(G) \otimes H$, for a separable Hilbert space $H$.
\end{lemma}
\begin{proof}
Let $H$ be the Hilbert space in the isomorphism 
\beq{eq iso L2GH}
L^2(E)\otimes L^2(G)\cong L^2(G)\otimes H
\eeq
from Theorem \ref{prop G rep}. Let $\mathcal{E}'_{C^*(G)}$ denote the Hilbert $C^*(G)$-module completion of $C_c(G)\otimes H$ with respect to the $C_c(G)$-valued inner product and right $C_c(G)$-action 
$$\langle s,t\rangle(g):=\langle s,gt\rangle_{L^2(G)\otimes H},\qquad s\cdot b:=\int_G g^{-1}(b(g)s)\,dg,$$ 
where $s,t\in L^2(G,H).$ Then the isomorphism \eqref{eq iso L2GH}, restricted to the dense subspace $C_c(G, L^2(E))\subseteq L^2(E)\otimes L^2(G)$, extends to an isomorphism $\mathcal{E}_{C^*(G)}\cong\mathcal{E}'_{C^*(G)}.$ Further, one can check that the map
$$\mathcal{E}'_{C^*(G)}\rightarrow C^*(G)\otimes H,\qquad s\mapsto\tilde{s},$$
where $\tilde{s}$ takes $g\mapsto s(g^{-1}),$ is an isometric isomorphism of $\mathcal{E}'_{C^*(G)}$ onto the standard Hilbert $C^*(G)$-module equipped with its usual inner product and right $C^*(G)$-action.
\end{proof}
Using Lemma \ref{lem:standardiso}, we can write down an identification $$U:\mathcal{K}(\mathcal{E}_{C^*(G)})\cong\mathcal{K}(C^*(G)\otimes H).$$

Now let $C^*(X)^G$ denote the  $G$-equivariant Roe algebra on $L^2(E)\otimes L^2(G)\cong L^2(G)\otimes H$, and let $C^*_{\textnormal{ker}}(X)^G$ be its dense subalgebra of $G$-invariant kernels from Definition \ref{def:admissiblekernelalgebra}. Let $W$ denote the identification
\[
W:C^*_{{\ker}}(X)^G\cong C_c(G)\otimes\mathcal{K}(H)
\]
below Proposition \ref{prop dense kernels}.
This map
identifies $C^*_{\ker}(X)^G$ with a subalgebra of the compact operators on the standard Hilbert $C^*(G)$-module $C^*(G)\otimes H$:
$$W:C^*_{\ker}(X)^G\xrightarrow{\cong}\underbrace{C_c(G)\otimes\mathcal{K}(H)}_{\subseteq\mathcal{K}\left(\mathcal{E}'_{C^*(G)}\right)}\xrightarrow{\cong}\underbrace{C_c(G)\otimes\mathcal{K}(H)}_{\subseteq\mathcal{K}\left(C^*(G)\otimes H\right)}.$$
This extends to an identification $W:C^*(X)^G\xrightarrow{\cong}\mathcal{K}(C^*(G)\otimes H)$. 

Let $\mathcal{M}$ be the multiplier algebra of $C^*(X)^G$, and let $\mathcal{L} := \mathcal{L}(\mathcal{E}_{C^*(G)})$ be the algebra of adjointable operators on $\mathcal{E}_{C^*(G)}$.
\begin{corollary}
\label{cor:roecompactoperators}
We have an isomorphism
$$U^{-1}\circ W:C^*(X)^G\xrightarrow{\cong}\mathcal{K}(\mathcal{E}_{C^*(G)}).$$
This induces an isomorphism on the multiplier algebras and an isomorphism on $K$-theory of the quotient algebras:
\[
(U^{-1}\circ W)_*:K_1(\mathcal{M}/C^*(X)^G)\xrightarrow{\cong}K_1(\mathcal{L}/\mathcal{K}(\mathcal{E}_{C^*(G)})).
\]
\end{corollary}

Now let $$\eta:K_0(\mathcal{K}(C^*(G)\otimes H))\rightarrow K_0(C^*(G))$$
be the stabilisation isomorphism on $K$-theory, and write 
$$\phi:=\eta\circ W_*,$$
where $W_*$ is the map on $K$-theory induced by $W$. After making these identifications, the following proposition follows directly from naturality of boundary maps with respect to $*$-homomorphisms.
\begin{proposition}
\label{prop:identifyingindices}
The following diagram commutes:
\begin{equation}
\xymatrix{
	K_{1}(\mathcal{M}/C^*(X)^G))\ar[dd]^{(U^{-1}\circ W)_*}\ar[rr]^{\ind} && K_{0}(C^*(X)^G)\ar[drr]^{\phi}&&\\
	&& && K_0(C^*(G)),\\
	K_{1}(\mathcal{L}/\mathcal{K}(\mathcal{E}_{C^*(G)}))\ar[rr]^{\ind} && K_{0}(\mathcal{K}(\mathcal{E}_{C^*(G)}))\ar[urr]^{\eta\circ U_*}&&\\
}
\end{equation}
where $U_*$ is the map induced by $U$ on $K$-theory.
\end{proposition}
The map $(\eta\circ U_*\circ\ind)$ is the usual $G$-index map for Fredholm operators in the sense of Hilbert $C^*(G)$-modules on the module $\mathcal{E}_{C^*(G)}$. Thus Proposition \ref{prop:identifyingindices} provides an identification of the index map in the Roe algebra picture with the usual notion of $G$-index for operators on a $G$-cocompact space.
\subsection{The localised equivariant index}
\label{sec ind maps noncocpt}
Now suppose that $X/G$ is possibly noncompact. As before,  let $E$ be a $G$-vector bundle over $X$. Similar to the previous subsection, equip the dense subspace $C_c(G,L^2(E))$ of $L^2(E)\otimes L^2(G)$ with the $C_c(G)$-valued inner product and right $C_c(G)$-action given by
$$\langle s,t\rangle(g):=\langle s,gt\rangle_{L^2(E)\otimes L^2(G)},\qquad s\cdot b:=\int_G g^{-1}(b(g)s)\,dg.$$
Taking the completion gives rise to a Hilbert $C^*(G)$-module $\mathcal{E}_{C^*(G)}.$

Let $Z\subset X$ be closed and $G$-invariant. Let $H_X$ be an admissible equivariant $C_0(X)$-module.
The restriction map $C_0(X)\to C_0(Z)$ allows us to view $H_X$ a  as a $C_0(Z)$-module, which will be degenerate (see Definition \ref{def C0X module}) in general. Let $H_Z=\one_Z H_X$, an admissible $C_0(Z)$-module. Write $H_{X\setminus Z}$ for the orthogonal complement of $H_Z$ in $H_X$. The map
\beq{eq phi incl B}
\varphi_Z^X\colon \cB(H_Z) \to \cB(H_X),
\eeq
defined by extending operators by zero on $H_{X\setminus Z}$, restricts to a $*$-homomorphism
\beq{eq phi incl Roe}
\varphi_Z^X\colon C^*(Z)^G\to C^*(X; H_X)^G,
\eeq
whose image lies in $C^*(X; Z, H_X)^G$.

Let $C^*(X)^G_{\loc}$ be the localised equivariant Roe algebra of Definition \ref{def loc Roe}. 
\begin{proposition}
We have
\[
C^*(X)^G_{\loc} \cong \cK(\cE_{C^*(G)}).
\]
\end{proposition}
\begin{proof}
Fix $Z\subseteq X$ a closed, $G$-invariant cocompact subset. For $i>0$, let $\Pen(Z,i)$ be as in \eqref{eq def Pen}. Then $\Pen(Z,i)$ is cocompact and $G$-stable, so by \eqref{eq red Roe cocpt},
%
Let $\varphi_i$ be the 
map
\[
\varphi_i := \varphi_{\Pen(Z,i)}^{\Pen(Z,i+1)}\colon C^*(\Pen(Z,i))^G \to C^*(\Pen(Z,i+1))^G
\]
as in \eqref{eq phi incl Roe}.
Then 
$$\{C^*(\Pen(Z,i))^G,\varphi_i\}_{i\in\mathbb{N}}$$
is a directed system of $C^*$-algebras whose direct limit is $C^*(X)^G_{\loc}$. Hence
$$C^*(X)^G_{\loc}\cong C^*(G)\otimes\mathcal{K}(H),$$ 
where $H$ 
is the Hilbert space from the isomorphism $L^2(E)\otimes L^2(G)\cong L^2(G)\otimes H$ in Definition \ref{def adm mod}. 

Now let $\mathcal{E}|_{\Pen(Z,i)}$ be the restriction of the Hilbert module $\mathcal{E}_{C^*(G)}$ to ${\Pen(Z,i)}$. By Corollary \ref{cor:roecompactoperators}, for each $i$, we have an isomorphism $\mathcal{K}(\mathcal{E}|_{\Pen(Z,i)})\cong C^*({\Pen(Z,i)})^G.$ These maps fit into a commutative diagram
\[
\xymatrix{
	C^*({\Pen(Z,i)})^G\ar[rr]^{\cong}\ar@{^{(}->}[d]^{\varphi_i}&& \mathcal{K}(\mathcal{E}|_{\Pen(Z,i)})\ar@{^{(}->}[d]\\
	C^*({\Pen(Z,i+1)})^G\ar[rr]^{\cong}&& \mathcal{K}(\mathcal{E}|_{\Pen(Z,i+1)}).
}
\]
Finally, note that each element of $\mathcal{K}(\mathcal{E}_{C^*(G)})$ is a limit of finite-rank operators, hence $\mathcal{K}(\mathcal{E}_{C^*(G)})=\lim_i\mathcal{K}\left(\mathcal{E}|_{\Pen(Z,i)}\right)$.
\end{proof}
It follows that we have an isomorphism $$\mathcal{L}(\mathcal{E}_{C^*(G)})/\mathcal{K}(\mathcal{E}_{C^*(G)})\cong \mathcal{M}(C^*(X)^G_{\loc})/C^*(X)^G_{\loc}.$$
Applying Proposition \ref{prop:identifyingindices} to each of the $G$-cocompact spaces ${\Pen(Z,i)}$ and taking the direct limit, we have shown:
\begin{proposition}
\label{prop:equalpictures}
The following two index maps are equal:
$$\ind_G:K_1(\mathcal{L}/\mathcal{K}(\mathcal{E}_{C^*(G)}))\xrightarrow{\lim(\eta\,\circ\,U_*\circ\,\ind)} K_0(C^*(G))$$
and
$$\ind_G:K_1(\mathcal{M}/(C^*(X)^G_{\loc}))\xrightarrow{\lim(\phi\,\circ\,\ind)}K_0(C^*(G)).$$
Here $C^*(G)$ and $C^*(X)^G_{\loc}$ can be taken to be either the reduced or maximal version of the group $C^*$ and Roe algebras.

\end{proposition}
\subsection{$G$-Callias-type operators and Roe's localised index}


We now relate the reduced version of the equivariant Callias-type index defined in \cite{Guo18} to (the reduced version of) the localised coarse index.

Recall the setting of Subsection \ref{sec Callias}, where $D = D_{\Phi}$ is a Callias-type operator. The operator $F$ in \eqref{eq F} defines a class $[F] \in K_1(\calL(\cE)/\cK(\cE))$, whose image under the boundary map for the six-term exact sequence corresponding to the ideal $\calL(\cE) \subset \cK(\cE)$ is by definition  $\ind_G^C(D_{\Phi}) \in K_0(C^*_{\red}(G))$.
Consider the embedding
$$\mathcal{E}\hookrightarrow\mathcal{E}_{C^*_{\red}(G)},$$
defined on the dense subspace $C_c(E)$ by the map $j$ in \eqref{eq def j}.
The image of $\mathcal{E}$ is a complemented submodule of $\mathcal{E}_{C^*_{\red}(G)}$. Extend the operator $F$ to all of $\mathcal{E}_{C^*_{\red}(G)}$ by defining the extension to be the identity on the orthogonal complement. We  denote this extended operator by $F_1$.

%
%
%

The assumption \eqref{eq Callias est} on $\Phi$ and $\tilde D$ implies that there is a $G$-cocompact subset $Z \subset X$ such that $D_{\Phi}^2 \geq c^2$ outside $Z$, for some $c>0$. By replacing $D_{\Phi}$ by the operator $\frac{1}{c} D_{\Phi}$, which has the same index as $D_{\Phi}$, we may alternatively make the slightly more convenient assumption that $D_{\Phi}^2 \geq 1$ outside $Z$.
As in Definition \ref{def loc index ell}, the localised equivariant coarse index of $D_{\Phi}$ is
\beq{eq loc index}
\ind_G(D_{\Phi}) = \ind_G([b(D_{\Phi})] \oplus 1) \quad \in K_0(C^*_{\red}(G)),
\eeq
for an odd, continuous function $b$ on $\mathbb{R}$ with
$$\textnormal{supp}(b^2-1)\subseteq [{-1},1],$$
We now make a specific choice for the function $b$:
\[
b(x) = \left\{\begin{array}{ll}
-1 & \text{if $x\in(-\infty,-1]$};\\
x &  \text{if $x\in(-1,1)$};\\
1 &  \text{if $x\in [1,\infty)$}.
\end{array}\right.
\]
%
%
We will write  $F_0:=b(D_{\Phi}) \oplus 1$, for this function $b$. Then we have a class
\[
[F_0]\in K_1((\mathcal{M}/C^*(X)^G_{\loc})).
\]
The index \eqref{eq loc index} equals the image of $[F_0]$ under the relevant boundary map,
\[
\partial[F_0] \in K_0(C^*(X)^G_{\loc}).
\]
Here the localised equivariant Roe algebra $C^*(X)^G_{\loc}$ is realised on the admissible module $L^2(E) \otimes L^2(G)$.

%
%

The operators $F_0$ and $F_1$ define elements of $K_1(\mathcal{M}/C^*(X)^G_{\loc})$ and $K_1(\mathcal{L}/\mathcal{K}(\mathcal{E}_{C^*_{\red}(G)}))$ respectively. By Proposition \ref{prop:equalpictures}, their indices in $K_0(C^*_{\red}(G))$ can be viewed equivalently through either of these pictures, and they equal the two sides of \eqref{eq thm Callias}. 
To prove Theorem \ref{thm Callias}, it therefore suffices to prove the following equality.
\begin{proposition} \label{prop F0 F1}
We have
\[
\ind_G
[F_0]= \ind_G[F_1] \quad \in K_0(C^*_{\red}(G))
\]
\end{proposition} 

\subsection{Proof of Theorem \ref{thm Callias}}

We now prove Proposition \ref{prop F0 F1}, and hence Theorem \ref{thm Callias}.

For $s > 0$, define the functions $b_s\in C_b(\mathbb{R})$ and $\psi_s \in C_0(\R)$ by
\[
\begin{split}
b_s(x)&=\frac{x}{(|x|^{1/s}+1)^{s}};\\
\psi_s(x)&=\frac{1}{(|x|^{1/s}+1)^{s}},
\end{split}
\]
for $x \in \R$. Then
\beq{eq approx bs b}
\lim_{s \downarrow 0} \|b_s - b\|_{\infty} = 0.
\eeq

Let $\zeta\colon (0,1]\to (0,1]$ be a continuous function such that $\zeta(1) =1$ and 
\beq{eq zeta s to zero}
\lim_{s\downarrow {0}}\zeta(s) \| \psi_{s/2}(D_{\Phi})\| = 0.
\eeq
For $s \in (0,1]$, consider the operator
\[
\tilde F_s := b_{s/2}(D_{\Phi}) + \zeta(s) \psi_{s/2}(D_{\Phi})\Phi
\]
on $\cE$.
For $s \in (0,1]$, the operator  $D_{\Phi}+\zeta(s)\Phi$ is of $G$-Callias type. Hence there is a 
 continuous, $G$-invariant, cocompactly supported function $f_s$ on $M$ such that
 \[
 \frac{D_{\Phi}+\zeta(s)\Phi}{\sqrt{(D_{\Phi}+\zeta(s)\Phi)^2+f_s}}
 \]
is invertible modulo $\cK(\cE)$. Since the operator $\sqrt{(D_{\Phi}+\zeta(s)\Phi)^2+f_s} \psi_{s/2}(D_{\Phi})$ is invertible, the operator
\[
\tilde F_s = \left(  \frac{D_{\Phi}+\zeta(s)\Phi}{\sqrt{(D_{\Phi}+\zeta(s)\Phi)^2+f_s}} \right)\sqrt{(D_{\Phi}+\zeta(s)\Phi)^2+f_s} \psi_{s/2}(D_{\Phi})
\]
is invertible modulo $\cK(\cE)$ as well. 

We have 
\[
 \tilde F_1 = \frac{D_{\Phi} + \Phi}{ \sqrt{(D_{\Phi}+\Phi)^2+f_1}} \frac{\sqrt{(D_{\Phi}+\Phi)^2+f_1}}{ \sqrt{(D_{\Phi}+\Phi)^2+1}}. 
\]
Hence $\tilde F_1$ has the same index as
\[
\frac{D_{\Phi} + \Phi}{  \sqrt{(D_{\Phi}+\Phi)^2+f_1}},
\]
which equals the index of $D_{\Phi}$.

Finally, \eqref{eq approx bs b} and \eqref{eq zeta s to zero} imply that
\[
\lim_{s\downarrow 0} \|\tilde F_s \oplus 1 - F_0\| = 0.
\]
So $s \mapsto \tilde F_s$ is a continous path of operators that are invertible modulo $\cK(\cE)$ connecting $F_0$ to the operator $\tilde F_1$, which has the same index in $K_0(C^*_{\red}(G))$ as $F_1$. This implies Proposition \ref{prop F0 F1}.

\section{Positive scalar curvature and the localised Baum--Connes conjecture} \label{sec appl}

\subsection{Positive scalar curvature}

\begin{proof}[Proof of Proposition \ref{prop psc}]
In the setting of the proposition, the operator $D$ has a well-defined localised index
\[
\ind_G^{Z}(D) \in K_*(C^*(M; Z)^G).
\]
 Then $\ind_G(D) \in K_*(C^*(M)^G)$ is the image of $\ind_G^{Z}(D)$ under the map
\[
K_*(C^*(Z)^G) = K_*(C^*(M; Z)^G) \to K_*(C^*(M)^G),
\]
and hence equal to zero.
\end{proof}

To prove Theorem \ref{thm:induction}, we use the following equivariant version of a theorem of Vilms \cite{Vilms} that was proved in \cite{GMW}.
	\begin{theorem}
		\label{thm:eq.Vilms}
		Let $\pi\colon M\to B$ be a fibre bundle with fibre $N$ and structure group $K$. 
		Suppose that $M$ and $B$ both have bounded geometry and proper, isometric $G$-actions making $\pi$ $G$-equivariant.
		Let $g_N$ be a $K$-invariant Riemannian metric on $N$. Then there is a $G$-invariant Riemannian metric $g_M$ on $M$ such that $\pi$ is 
		a $G$-equivariant Riemannian submersion with totally geodesic fibres.
	\end{theorem}

	\begin{proof}[Proof of Theorem \ref{thm:induction}] Let $\kappa_{G/K}$ denote the scalar curvature of the  $G$-invariant  Riemannian metric $g_{G/K}$ on the base of the fibre bundle $M \to G/K$. Note that since $G/K$ is a homogeneous space, $\kappa_{G/K}$ is a finite constant. Let $H\subseteq TM$ be an Ehresmann connection. Then as in the proof of Theorem \ref{thm:eq.Vilms} above, we may lift $g_{G/K}$ to a $G$-invariant metric $g_H$ on $H$, as well as lift the $K$-invariant Riemannian metric $g_N$ on $N$ to a metric on the vertical subbundle $V\subseteq TM$. Define a $G$-invariant metric on $M$ by $g_M := g_H\oplus g_V$. 
		
		Since $N$ has uniformly positive scalar curvature $\kappa_N$, it satisfies $\inf\{\kappa_N\}=:\kappa_0>0$.
		Now let $T$ and $A$ denote the O'Neill tensors of the submersion $\pi$ (their definitions can be found in \cite{ONeill66}). By Theorem \ref{thm:eq.Vilms} above, the fibres of $M$ are totally geodesic, so $T=0$. Pick an orthonormal basis of horizontal vector fields $\{X_i\}$.
		A result of Kramer (\cite{Kramer00},
		p.\ 596),  relates the scalar curvatures by
		$$
		\kappa_M(p) = \kappa_{G/K} + \kappa_N(p) - \sum_{i,j} ||A_{X_i}(X_j)||_p.
		$$
		Since both $M$ and $N$ have bounded geometry, it follows that their scalar curvatures $\kappa_M$ and $ \kappa_N$ are uniformly bounded.
		Therefore 		
\[
\sup_{p\in M}\sum_{i,j} ||A_{X_i}(X_j)||_p \le A_0<\infty
\]
 for some positive constant $A_0$.
		Upon scaling the fibre metric on $N$ by a positive factor $t$, we obtain
		$$
		\kappa_M(p) \ge \kappa_{G/K} + t^{-2}\kappa_0 - A_0 >0 \qquad \text{whenever}\quad 0<t< \sqrt{\frac{\kappa_0}
			{-\kappa_{G/K} + A_0}},
		$$
		where we choose $A_0>0$ large enough such that $-\kappa_{G/K} + A_0>0$.
		Thus $g_M$ is a $G$-invariant metric of uniform positive scalar curvature on $M$.
	\end{proof}

\subsection{The localised Baum--Connes conjecture}


Let $Z\subset X$ be a closed, $G$-invariant, cocompact subset. Let $H_X$ be an admissible equivariant $C_0(X)$-module.
The map $\varphi_Z^X$ in \eqref{eq phi incl B} restricts to a map
\[
\varphi_Z^X\colon D^*_{\red}(Z)^G\to D^*_{\red}(X)^G
\]
that maps $C^*(Z)^G$ into $C^*(X; Z)^G$.  Hence we obtain
\beq{eq jstar}
\varphi_Z^X\colon D^*_{\red}(Z)^G/C^*(Z)^G \to D^*_{\red}(X)^G/C^*(X; Z)^G.
\eeq
By Paschke duality, the analytic $K$-homology of $Z$ equals
\[
K_*^G(Z) = K_{*+1}(D^*_{\red}(Z)^G/C^*(Z)^G).
\]
So \eqref{eq jstar} induces
\beq{eq j star}
(\varphi_Z^X)_*\colon K^G_*(Z)\to K^G_*(X)_{\loc}.
\eeq

Using the maps \eqref{eq j star}, we obtain
\beq{eq jstar glob}
\varphi^X_*\colon RK_*^G(X) \to K^G_*(X)_{\loc}.
\eeq
\begin{lemma}
The following diagram commutes, where $\mu^X_G$ denotes the analytic assembly map for $X$:
\[
\xymatrix{
RK_*^G(X)\ar[r]^-{\mu^X_G} \ar[d]_-{\varphi^X_*}& K_*(C_{\red}^*G). \\
K^G_*(X)_{\loc}\ar[ur]_-{\ind^{\loc}_G}
}
\]
\end{lemma}
\begin{proof}
If $Z\subset X$ is $G$-cocompact, then naturality of boundary maps with respect to $*$-homomorphisms implies that the diagram
\[
\xymatrix{
K_*^G(Z)\ar[r]^-{\ind_G} \ar[d]_-{(\varphi_Z^X)_*}& K_*(C_{\red}^*G). \\
K^G_*(X)_{\loc} \ar[ur]_-{\ind^{\loc}_G}
}
\]
commutes.  By Corollary \ref{cor loc index ass map}, the top horizontal arrow equals $\mu^Z_G$, so the claim follows after we take direct limits. 
\end{proof}
This lemma directly implies Proposition \ref{prop LBCC}.

\begin{remark}
\label{rem BC inj surj}
The arguments in this subsection have two more consequences.
\begin{enumerate}
\item If the map \eqref{eq jstar glob} is injective, then injectivity of the Baum--Connes assembly map implies injectivity of the map in Conjecture \ref{conj LBCC}.
\item If the map \eqref{eq jstar glob} is surjective, then Conjecture \ref{conj LBCC} implies surjectivity
 of the Baum--Connes assembly map.
\end{enumerate}
\end{remark}

\bibliographystyle{plain}

\bibliography{mybib}

\end{document}